%% file: spectrum.tex
\documentclass[12pt,a4paper,draft]{amsart}
\input{pream}\input{defin}
\begin{document}

\title[Borel structure of the spectrum]{Borel structure of the spectrum\\of a closed operator}
\myData
\begin{abstract}
For a linear operator $T$ in a Banach space let $\sigma_p(T)$ denote the point spectrum of $T$, $\sigma_{p[n]}(T)$
for finite $n > 0$ be the set of all $\lambda \in \sigma_p(T)$ such that $\dim \ker (T - \lambda) = n$ and let
$\sigma_{p[\infty]}(T)$ be the set of all $\lambda \in \sigma_p(T)$ for which $\ker (T - \lambda)$ is infinite-dimensional.
It is shown that $\sigma_p(T)$ is $\FFf_{\sigma}$, $\sigma_{p[\infty]}(T)$ is $\FFf_{\sigma\delta}$ and for each finite
$n$ the set $\sigma_{p[n]}(T)$ is the intersection of an $\FFf_{\sigma}$ and a $\GGg_{\delta}$ set provided $T$ is closable
and the domain of $T$ is separable and weakly $\sigma$-compact. For closed densely defined operators in a separable Hilbert
space $\HHh$ more detailed decomposition of the spectra is done and the algebra of all bounded linear operators on $\HHh$
is decomposed into Borel parts. In particular, it is shown that the set of all closed range operators on $\HHh$ is Borel.
\end{abstract}
\maketitle


\SECT{Introduction}

Let $\HHh$ be a separable Hilbert space. As it is easily seen, any subset of the complex plane is the point spectrum
of an unbounded linear operator acting on $\HHh$. To see this, note that the dimension of $\HHh$ as a vector space is equal
to the power of the continuum; take a Hamel basis $\{e_s\}_{s \in \RRR}$ of $\HHh$ and any function $\psi\dd \RRR
\to \CCC$, define $A\dd \HHh \to \HHh$ by the rule $A e_s = \psi(s) e_s$, and observe that the point spectrum of $A$
coincides with the image of the function $f$ (see also \cite[Example~2]{j-s}). So, non-Borel subsets of $\CCC$ may be
the point spectra of certain linear operators acting on $\HHh$. It is also worth while to mention that every bounded subset
of $\CCC$ is the point spectrum of a bounded normal (even diagonal) operator on a nonseparable Hilbert space. However,
all densely defined linear operators in separable Hilbert spaces whose point spectra are non-Borel which appear in all
existing examples in the literature are nonclosed. It was therefore quite natural to ask for such an example in the class
of closed operators. In the recent paper we will show that such an example does not exist. More generally, we will prove
the following result (see \THM{wcompact} and {\THM{general}):
\begin{quote}
\textit{Let $X$ be a (real or complex) Banach space and $T$ be a closable operator in $X$ whose domain is separable
and weakly $\sigma$-compact. Then the point spectrum of $T$ is $\FFf_{\sigma}$ and for each $n \geqsl 1$ the set of all
scalars $\lambda$ such that the kernel of $T - \lambda$ is $n$-dimensional is $\FFf_{\sigma\delta}$.}
\end{quote}
In particular, the above theorem applies to the point spectrum of a closed operator in $\HHh$, which solves the problem
posed in \cite{j-s} (see Question after Example 2 there). Partial solution of this issue for certain classes of closed
operators in Hilbert spaces may be found e.g. in \cite[Corollary~7]{s-sz}. In \EXM{bounded} we give a simple example
of a bounded densely defined operator in $\HHh$ whose point spectrum coincides with an arbitrarily given uncountable subset
of the unit disc. The problem whether the point spectrum of a bounded operator in a separable nonreflexive Banach space
is $\FFf_{\sigma}$ we leave as open.\par
Since the celebrated paper by Cowen and Douglas \cite{c-d} was published, the point spectra of bounded linear operators are
widely investigated and those operators which have these spectra rich have attracted a growing interest,
see e.g. \cite{mc}, \cite{zhu}, \cite{jiang} or \cite{j-s} and references there.

\textbf{Notation.} In this paper Banach spaces are real or complex and they need not be separable, and $\KKK$ is the field
of real or complex numbers. By an \textit{operator in} a Banach space $X$ we mean any linear function from a linear
subspace of $X$ into $X$ and such an operator is \textit{on} $X$ if its domain is the whole space. For an operator $T$
in $X$, $\DdD(T)$, $\NnN(T)$, $\RrR(T)$ and $\overline{\RrR}(T)$ denote the domain, the kernel, the range and the closure
of the range of $T$, respectively. For a scalar $\lambda \in \KKK$, $T - \lambda$ stands for the operator given
by the rules: $\DdD(T - \lambda) = \DdD(T)$ and $(T - \lambda)(x) = Tx - \lambda x$ ($x \in \DdD(T)$). The \textit{point
spectrum} of $T$, $\sigma_p(T)$, is the set of all eigenvalues of $T$; that is, $\sigma_p(T)$ consists of all scalars
$\lambda \in \KKK$ such that $\NnN(T - \lambda)$ is nonzero. The operator $T$ is \textit{closed} iff its graph $\Gamma(T)
:= \{(x,Tx)\dd\ x \in \DdD(T)\}$ is a closed subset of $X \times X$. $T$ is \textit{closable} iff the norm closure
of $\Gamma(T)$ in $X \times X$ is the graph of some operator in $X$. Whenever $E$ and $F$ are Banach spaces, $\BBb(E,F)$
stands for the space of all bounded operators from the whole space $E$ to $F$. A subset of a topological space is said to
be \textit{$\sigma$-compact} if it is the union of a countable family of compact subsets of the space.
A \textit{weakly $\sigma$-compact} subset of a Banach space is a set which is $\sigma$-compact with respect to the weak
topology of the space. A \textit{Borel} subset of a topological space is a member of the $\sigma$-algebra generated by all
open sets in the space. Adapting the notation proposed by A.H. Stone \cite{ahstone}, a subset of a topological space which
is the intersection of an open and a closed set is called by us (\textit{of type}) \textit{$\FFf \cap \GGg$}. Similarly,
any set being the intersection of a $\GGg_{\delta}$ and an $\FFf_{\sigma}$ set will be said to be
$\FFf_{\sigma} \cap \GGg_{\delta}$.

\SECT{Arbitrary closed operator in a Banach space}

In this and next sections $X$ denotes a Banach space.\par
Recall that a \textit{Souslin} set is a metrizable space which is the continuous image of a separable complete metric
space (see the beginning of Chapter~XIII of \cite{k-m} and Theorem~XIII.1.1 there, or Appendix in \cite{takesaki}). Souslin
subsets of separable complete metric spaces are \textit{absolutely measurable}, i.e. whenever $A$ is a Souslin subset
of a separable complete metric space $Y$ and $\mu$ is a nonnegative $\sigma$-finite Borel measure on $Y$, there are Borel
subsets $B$ and $C$ of $Y$ such that $B \subset A \subset C$ and $\mu(C \setminus B) = 0$
(see e.g. \cite[Theorem~A.13]{takesaki}).\par
We begin with

\begin{pro}{souslin}
Let $T$ be a closable operator in $X$ whose domain is the image of a separable Banach space under a bounded operator.
The point spectrum of $T$ is a Souslin subset of $\KKK$. In particular, $\sigma_p(T)$ is absolutely measurable.
\end{pro}
\begin{proof}
Let $E$ be a separable Banach space and $S \in \BBb(E,X)$ be such that $S(E) = \DdD(T)$. Passing to the quotient of $E$
by $\NnN(S)$, we may assume that $S$ is one-to-one. Since $T$ is closable, the operator $C = TS\dd E \to X$ is bounded
(by the Closed Graph Theorem). Notice that $\sigma_p(T) = \{\lambda \in \KKK\dd\ \NnN((T - \lambda)S) = \NnN(C - \lambda S)
\neq \{0\}\}$. Let $W$ be the set of all $x \in E$ such that $\|x\| = 1$ and $Cx = \lambda(x) Sx$ for some $\lambda(x)
\in \KKK$. We have obtained a function $\lambda\dd W \to \KKK$ with $\lambda(W) = \sigma_p(T)$. So, to finish the proof
it suffices to show that $W$ is closed and $\lambda$ is continuous.\par
Suppose $x_n \in W$ and $x_n \to x \in X$ as $n \to \infty$. Then $\|x\| = 1$ and hence $Sx \neq 0$. We infer from this
that the sequence $(\lambda(x_n))_{n=1}^{\infty}$ is bounded because $|\lambda(x_n)| = \frac{\|Cx_n\|}{\|Sx_n\|}
\to \frac{\|Cx\|}{\|Sx\|}\ (n \to \infty)$. Now if $(\lambda(x_{n_k}))_{k=1}^{\infty}$ is any subsequence which converges
to some $w \in \KKK$, then $Cx = w Sx$. This shows that $x \in W$ and $w = \lambda(x)$. So, $W$ is closed and $\lambda$
is continuous.
\end{proof}

Now since the domain of a closed operator is the image of its graph under the projection onto the first factor,
\PRO{souslin} applies to closed operators in separable Banach spaces and gives

\begin{cor}{cl-Souslin}
The point spectrum of a closed operator in a separable Banach space is Lebesgue measurable.
\end{cor}

In the next section (\COR{reflexive}) we shall prove that the point spectrum of a closed operator in a separable
reflexive Banach space is Borel (even $\FFf_{\sigma}$). We do not know whether the assumption of reflexivity may
be relaxed in this.\\
\nl
\textbf{Question 1.} Is the point spectrum of a bounded operator on a separable Banach space Borel?

\SECT{Operators with weakly $\sigma$-compact domains}

For an operator $T$ in $X$, a subset $K$ of $X$ and a nonnegative real constant $M$, let
$$
\Lambda_T(K,M) = \{w \in \KKK\dd\ \NnN(T - w) \cap K \neq \varempty \textup{ and } |w| \leqsl M\}.
$$
Notice that $\Lambda_T(K,M)$ consists of all $w \in \KKK$ with $|w| \leqsl M$ provided $0 \in K$. The main tool
of this section is the following

\begin{lem}{wcompact}
If $T$ is a closable operator in $X$ and $K$ is a weakly compact subset of $\DdD(T)$, then the set $\Lambda_T(K,M)$
is compact for every $M \geqsl 0$.
\end{lem}
\begin{proof}
We may and do assume that $0 \notin K$. Let $W$ be the set of all $x \in K$ for which there is (unique)
$\lambda(x) \in \KKK$ such that $Tx = \lambda(x) x$ and $|\lambda(x)| \leqsl M$. Thus we have obtained a function
$\lambda\dd W \to \KKK$. Since $\lambda(W) = \Lambda_T(K,M)$, it suffices to show that $W$ is weakly compact
and $\lambda$ is continuous when $W$ is equipped with the weak topology.\par
Let $\XxX = (x_{\sigma})_{\sigma \in \Sigma}$ be a net in $W$ which is weakly convergent to some $x \in K$. We need to show
that $x \in W$ and $\lim_{\sigma \in \Sigma} \lambda(x_{\sigma}) = \lambda(x)$. If $(x_{\tau})_{\tau \in \Sigma'}$
is a subnet of $\XxX$ such that $\lim_{\tau \in \Sigma'} \lambda(x_{\tau}) = w \in \KKK$,
then $(x_{\tau},Tx_{\tau})_{\tau \in \Sigma'}$ is weakly convergent to $(x,w x)$. Since the norm closure of $\Gamma(T)$
is weakly closed (and is the graph of some operator), we infer from this that $Tx = wx$ and thus $x \in W$
and $\lambda(x) = w$. Since $\lambda(W)$ is bounded and any convergent subnet
of $(\lambda(x_{\sigma}))_{\sigma \in \Sigma}$ has the same limit, the latter net converges to $\lambda(x)$ and we are
done.
\end{proof}

With use of the foregoing result we now easily prove the following

\begin{pro}{wcompact-0}
If $T$ is a closable operator in $X$ such that $\DdD(T) \setminus \{0\}$ is weakly $\sigma$-compact, then $\sigma_p(T)$
is $\FFf_{\sigma}$.
\end{pro}
\begin{proof}
Write $\DdD(T) \setminus \{0\} = \bigcup_{n=1}^{\infty} K_n$ with each $K_n$ weakly compact, note that
$\lambda \in \sigma_p(T)$ iff $\lambda \in \Lambda_T(K_n,m)$ for some natural $n$ and $m$ and apply \LEM{wcompact}.
\end{proof}

For applications of the above result, we need the next well known fact. For reader convenience, we give its short proof.

\begin{lem}{sep-wcomp}
Every weakly compact subset of $X$ which is separable in the norm topology is weakly metrizable.
\end{lem}
\begin{proof}
Let $K$ be a separable weakly compact subset of $X$. Then $(K - K) \setminus \{0\}$ is separable as well and thus there
is a sequence of linear functionals $f_n\dd X \to \KKK$ of norm $1$ such that for every nonzero $z \in K - K$ there is
$n$ with $f_n(z) \neq 0$. Observe that then the family $\{f_n\}_{n \in \NNN}$ separates the points of $K$. Finally, since
$K$ is weakly compact, the formula $K \ni x \mapsto (f_n(x))_{n \in \NNN} \in \Delta^{\NNN}$ with $\Delta =
\{w \in \KKK\dd\ |w| \leqsl 1\}$ defines a topological embedding of $K$ (equipped with the weak topology) into the compact
metrizable space $\Delta^{\NNN}$.
\end{proof}

Now we have

\begin{pro}{sep-wcomp}
If $\DDd$ is a linear subspace of $X$, then $\DDd \setminus \{0\}$ is weakly $\sigma$-compact iff $\DDd$ is separable
and weakly $\sigma$-compact as well.
\end{pro}
\begin{proof}
First assume that $\DDd$ is separable and $\DDd = \bigcup_{n=1}^{\infty} K_n$ with each $K_n$ weakly compact.
By \LEM{sep-wcomp}, $K_n$ is weakly metrizable and thus $L_n := K_n \setminus \{0\}$ is an $\FFf_{\sigma}$-subset of $K_n$
with respect to the weak topology. So, $\DDd \setminus \{0\} = \bigcup_{n=1}^{\infty} L_n$ and each $L_n$ is weakly
$\sigma$-compact.\par
Conversely, if $\DDd \setminus \{0\}$ is weakly $\sigma$-compact, so is $\DDd$ and $\{0\}$ is weakly $\GGg_{\delta}$
in $\DDd$. Thus, by the definition of the weak topology, there are sequences $(f_n)_{n=1}^{\infty}$
and $(\epsi_n)_{n=1}^{\infty}$ of continuous linear functionals on $X$ and of positive real numbers (respectively)
such that $\{0\} = \DDd \cap \{x \in X\dd\ |f_n(x)| < \epsi_n,\ n=1,2,\ldots\}$. In particular, $\bigcap_{n=1}^{\infty}
\NnN(f_n) \cap \DDd = \{0\}$ and therefore the function $\psi\dd \DDd \ni x \mapsto (f_n(x))_{n=1}^{\infty} \in
\KKK^{\NNN}$ is one-to-one. What is more, $\psi$ is continuous with respect to the weak topology on $\DDd$ and the product
one on $\KKK^{\NNN}$. So, $\psi$ restricted to any weakly compact subset $K$ of $\DDd$ is a topological embedding
and therefore every such $K$ is weakly separable. We conclude from this that $\DDd$ itself is weakly separable,
being weakly $\sigma$-compact. Finally, if $A$ is a countable weakly dense subset of $\DDd$ and $\EEe$ is the norm closure
of the linear span of $A$, then $\EEe$ is separable and weakly closed. The latter implies that $\DDd \subset \EEe$
and we are done.
\end{proof}

It follows from \PRO{sep-wcomp} that \PRO{wcompact-0} is equivalent to

\begin{thm}{wcompact}
If $T$ is a closable operator in $X$ whose domain is separable and weakly $\sigma$-compact, then $\sigma_p(T)$
is $\FFf_{\sigma}$.
\end{thm}

By Baire's theorem, a Banach space is weakly $\sigma$-compact iff it is reflexive. Thus \THM{wcompact} applies
mainly to reflexive spaces. For example:

\begin{cor}{imrefl}
If $T$ is a closable operator in $X$ whose domain is separable and is the image of a reflexive Banach space under a bounded
operator, then $\sigma_p(T)$ is $\FFf_{\sigma}$.
\end{cor}

Now argument used in the note just before \COR{cl-Souslin} shows that

\begin{cor}{reflexive}
The point spectrum of a closed operator in a separable reflexive Banach space is $\FFf_{\sigma}$. In particular, the point
spectrum of a closed densely defined operator in a separable Hilbert space is $\FFf_{\sigma}$.
\end{cor}

The second statement of the above result answers in the affirmative the question posed by Jung and Stochel in \cite{j-s}.

\begin{rem}{other}
\COR{reflexive} for bounded operators may also be proved in a different way. Namely, if $T$ is a bounded operator
on a separable reflexive Banach space $X$ (respectively a closed densely defined operator in a separable Hilbert space
$\HHh$), then, by the reflexivity of $X$, $\NnN(T) = \{0\}$ iff the range of the dual operator $T'\dd X' \to X'$
(respectively iff the range of the adjoint operator $T^*$) is dense in $X'$ (in $\HHh$). So, $\sigma_p(T)$ consists
of precisely those $\lambda \in \KKK$ for which $\RrR(S_{\lambda})$ is \textbf{not} dense in $X'$ (in $\HHh$) where
$S_{\lambda} = T' - \lambda$ ($S_{\lambda} = T^* - \bar{\lambda}$). Consequently, if $\{y_n\}_{n=1}^{\infty}$ is a dense
sequence in $X'$ (in $\HHh$), then
$$
\sigma_p(T) = \bigcup_{n,m \geqsl 1} \bigcap_{x \in \DdD(S_0)} F(n,m,x)
$$
with $F(n,m,x) := \{\lambda \in \KKK\dd\ \|S_{\lambda}x - y_n\| \geqsl \frac1m\}$. We infer from the closedness
of $F(n,m,x)$'s that indeed $\sigma_p(T)$ is $\FFf_{\sigma}$. This elementary and simpler proof does not work in case
when $T$ is only closable. Moreover, this approach hides the crucial property of reflexivity for this issue revealed
by \LEM{wcompact}, namely the weak $\sigma$-compactness of the domain of an operator.
\end{rem}

\begin{exm}{bounded}
Let $\HHh$ be a complex separable infinite-dimensional Hilbert space and let $S\dd \HHh \to \HHh$ be the backward shift
with respect to the orthonormal basis $\EEe = (e_n)_{n=0}^{\infty}$ of $\HHh$; that is, $S e_0 = 0$ and $S e_n = e_{n-1}$
for $n \geqsl 1$. Let $\DDD$ be the open unit disc in $\CCC$ and $\psi\dd \DDD \ni z \mapsto \sum_{n=0}^{\infty} z^n e_n
\in \HHh$. Recall that $\sigma_p(S) = \DDD$, $S \psi(z) = z \psi(z)$ for each $z \in \DDD$ and the image of $\psi$
is linearly independent. What is more, the linear span $\DDd_F$ of $\psi(F)$ is dense in $\HHh$ for every uncountable
subset $F$ of $\DDD$. (To see the latter, note that if $x$ is orthogonal to $\psi(F)$, then $\sum_{n=0}^{\infty} x_n z^n =
0$ for each $z \in F$ where $x_n$ is the $n$-th Fourier coefficient of $x$ in the basis $\EEe$. Thus $x_n = 0$ for any $n$
and $x = 0$ as well.) So, the restriction $S_F$ of $S$ to $\DDd_F$ is a bounded densely defined operator in $\HHh$
such that $\sigma_p(S_F) = F$. The example shows that in \THM{wcompact} the domain of a closable operator has to have
additional properties beside the separability.
\end{exm}

\begin{exm}{Linf}
The inclusion map of $L^{\infty}[0,1]$ into $L^2[0,1]$, as the dual operator of the inclusion map of $L^2[0,1]$
into $L^2[0,1]$, is continuous with respect to the weak-* topology of the domain and the weak one in $L^2[0,1]$.
This implies that $L^{\infty}[0,1]$, considered as a subspace of $L^2[0,1]$, is weakly $\sigma$-compact. However,
there is no closed operator in $L^2[0,1]$ whose domain is $L^{\infty}[0,1]$ (compare with the remark on page~257
in \cite{ranges}). The example shows that \THM{wcompact} can have quite natural applications also for nonclosed operators
and that it is more general than \COR{imrefl} which does not apply here, since there is no bounded operator on a reflexive
Banach space into $L^2[0,1]$ whose image is $L^{\infty}[0,1]$ (again, by the remark on page~257 in \cite{ranges}).
\end{exm}

\SECT{Decomposition of the point spectrum}

Let $T$ be an operator in $X$. For $n \geqsl 1$ let $\sigma_{p[n]}(T)$ be the set of all $\lambda \in \KKK$ such that
$\dim \NnN(T - \lambda) = n$ and let $\sigma_{p[\infty]}(T) = \sigma_p(T) \setminus \bigcup_{n=1}^{\infty}
\sigma_{p[n]}(T)$. Our aim is to show that all just defined sets are Borel provided $T$ is closable and has separable
and weakly $\sigma$-compact domain. To do this, we need

\begin{lem}{lin-depend}
If $V$ is a $T_2$ topological vector space, then for each $n \geqsl 2$ the set $F[n]$ of all $(x_1,\ldots,x_n) \in V^n$
such that $x_1,\ldots,x_n$ are linearly dependent is closed in the product topology of $V^n$.
\end{lem}
\begin{proof}
Let $\Delta = \{\lambda \in \KKK\dd\ |\lambda| \leqsl 1\}$. For a permutation $\tau$ of $\{1,\ldots,n\}$ let $F_{\tau}$
be the set of all $n$-tuples $(x_1,\ldots,x_n) \in V^n$ such that $x_{\tau(1)} = \sum_{j=2}^n \lambda_j x_{\tau(j)}$
for some $\lambda_2,\ldots,\lambda_n \in \Delta$. Notice that $F[n] = \bigcup_{j=1}^n F_{\tau_j}$ where $\tau_j(1) = j$,
$\tau_j(j) = 1$ and $\tau_j(k) = k$ for $k \neq 1,j$. So, it suffices to show that $F_{\tau}$ is closed. But $F_{\tau} =
p(\Phi^{-1}(\{0\})$ where $p\dd \Delta^{n-1} \times V^n \to V^n$ is the projection onto the second factor and
$$
\Phi\dd \Delta^{n-1} \times V^n \ni (\lambda_2,\ldots,\lambda_n;x_1,\ldots,x_n) \mapsto
x_1 - \sum_{j=2}^n \lambda_j x_j \in V.
$$
Since $\Delta^{n-1}$ is compact, $p$ is a closed map and we are done.
\end{proof}

The main result of the section is the following

\begin{thm}{general}
If $T$ is a closable operator in $X$ whose domain is separable and weakly $\sigma$-compact, then $\sigma_{p[n]}(T)$
is $\FFf_{\sigma} \cap \GGg_{\delta}$ for finite $n$ and $\sigma_{p[\infty]}(T)$ is $\FFf_{\sigma\delta}$.
\end{thm}
\begin{proof}
First fix finite $N > 0$. Let $F[1] = \{0\} \subset X$ and $F[N]$ be as in the statement of \LEM{lin-depend} for $N > 1$.
By \LEM{lin-depend}, $F[N]$ is weakly closed in $X^N$. Write $\DdD(T) = \bigcup_{n=1}^{\infty} K_n$ with each $K_n$ weakly
compact. By \LEM{sep-wcomp}, all $K_n$'s are weakly metrizable and hence $(\prod_{j=1}^N K_{n_j}) \setminus F[N]$ is
$\FFf_{\sigma}$ in $\prod_{j=1}^N K_{n_j}$ for any $n_1,\ldots,n_N$ when each $K_n$ is equipped with the weak topology.
So, $\DdD(T)^N \setminus F[N]$ may be written in the form
\begin{equation}\label{eqn:Ln}
\DdD(T)^N \setminus F[N] = \bigcup_{n=1}^{\infty} L_n
\end{equation}
where each $L_n$ is a weakly compact subset of $X^N$. Put $S\dd \DdD(T)^N \ni (x_1,\ldots,x_N) \mapsto (Tx_1,\ldots,Tx_N)
\in X^N$ and observe that $S$ is a closable operator in $X^N$. Thanks to \LEM{wcompact}, the set $\Lambda_S(L_n,m)$
is compact for all natural $n$ and $m$. So, $G_N := \bigcup_{n,m} \Lambda_S(L_n,m)$ is $\FFf_{\sigma}$. But $G_N$ is
the set of all $\lambda \in \sigma_p(T)$ such that $\NnN(T - \lambda)$ is at least$N$-dimensional, by \eqref{eqn:Ln}.\par
Now observe that for finite $n$, $\sigma_{p[n]}(T) = G_n \setminus G_{n+1}$ and thus $\sigma_{p[n]}(T)$ is
$\FFf_{\sigma} \cap \GGg_{\delta}$. Finally, since $\sigma_{p[\infty]}(T) = \bigcap_{n=1}^{\infty} G_n$,
$\sigma_{p[\infty]}(T)$ is $\FFf_{\sigma\delta}$.
\end{proof}

\SECT{Decomposition of the spectrum: Hilbert space}

From now on, $\HHh$ is a complex separable infinite-dimensional Hilbert space and $\BBb(\HHh) = \BBb(\HHh,\HHh)$.
Every subset of $\BBb(\HHh)$ which is Borel with respect to, respectively, the weak operator, the strong operator
or the norm topology is called by us, respectively, \textit{WOT-Borel}, \textit{SOT-Borel} and shortly \textit{Borel}.
Similarly, if $f\dd S \to \BBb(\HHh)$ with $S \subset \BBb(\HHh)$ is such that $S$ and the inverse image of any WOT-Borel
(respectively SOT-Borel; Borel) set is WOT-Borel (SOT-Borel; Borel) is said to be \textit{WOT-Borel} (\textit{SOT-Borel};
\textit{Borel}). A fundamental result in this topic says that every WOT-Borel set is SOT-Borel and conversely
(see e.g. \cite{effros}). Thus the same property holds true for WOT-Borel and SOT-Borel functions. Therefore we shall only
speak on WOT-Borel and Borel sets and functions. Whenever in the sequel the classes $\FFf_{\sigma}$, $\GGg_{\delta}$,
$\FFf \cap \GGg$, etc., appear, they are understood with respect to the norm topology.\par
We are interested in the decomposition of $\BBb(\HHh)$ into Borel sets each of which collects the operators of a similar
type. For $n = 0,1,2,\ldots$ let $\Sigma_n(\HHh)$ be the set of all operators $A \in \BBb(\HHh)$ such that $\dim \RrR(A) =
n$. Further, for $n,m = 0,1,2,\ldots,\infty$ let, respectively, $\Sigma_{n,m}^1(\HHh)$ and $\Sigma_{n,m}^0(\HHh)$
consist of all $A \in \BBb(\HHh)$ such that $\dim \NnN(A) = n$, $\dim \RrR(A)^{\perp} = m$, $\dim \overline{\RrR}(A) =
\infty$ and, respectively, $\RrR(A)$ is closed or not. Notice that all just defined sets of operators form a countable
decomposition of $\BBb(\HHh)$, denoted by $\Sigma(\HHh)$, into nonempty and pairwise disjoint sets. We want to show
that each of them is Borel in $\BBb(\HHh)$. To do this, we need

\begin{lem}{ker=n}
For each $n=0,1,2,\ldots,\infty$, let $\Sigma_{n,*}(\HHh)$ be the set of all operators $A \in \BBb(\HHh)$ such that
$\dim \NnN(A) = n$. Then $\Sigma_{0,*}(\HHh)$ is $\GGg_{\delta}$, $\Sigma_{n,*}(\HHh)$ is $\FFf_{\sigma} \cap
\GGg_{\delta}$ for finite $n > 0$ and $\Sigma_{\infty,*}(\HHh)$ is $\FFf_{\sigma\delta}$ in $\BBb(\HHh)$.
\end{lem}
\begin{proof}
We mimic the proof of \THM{general}. For fixed finite $N > 0$ let $F[N]$ be defined as there. Write
$\HHh^N \setminus F[N] = \bigcup_{n=1}^{\infty} L_n$ with each $L_n$ weakly compact in $\HHh^N$. For $T \in \BBb(\HHh)$
let $T^{\times N} \in \BBb(\HHh^N)$ denote the operator $\HHh^N \ni (x_1,\ldots,x_N) \mapsto (Tx_1,\ldots,Tx_N)
\in \HHh^N$. Since $L_n$ is bounded, the function $\psi\dd \BBb(\HHh) \times L_n \ni (T,x) \mapsto T^{\times N}x
\in \HHh^N$ is continuous when $\BBb(\HHh)$ is considered with the norm topology and $L_n$ and $\HHh^N$ with the weak one.
Hence $\psi^{-1}(\{0\})$ is closed in $\BBb(\HHh) \times L_n$ in the product topology of these topologies. Finally, since
$L_n$ is weakly compact, the set $F_n := p(\psi^{-1}(\{0\}))$ is closed in the norm topology of $\BBb(\HHh)$
where $p\dd \BBb(\HHh) \times L_n \to \BBb(\HHh)$ is the projection onto the first factor.
Thus $G_N := \bigcup_{n=1}^{\infty} F_n$ is $\FFf_{\sigma}$. Note that $G_N$ coincides with the set of all
$A \in \BBb(\HHh)$ for which $\dim \NnN(A) \geqsl N$.\par
Now we have $\Sigma_{n,*}(\HHh) = G_n \setminus G_{n+1}$ for finite $n > 0$, $\Sigma_{\infty,*}(\HHh) =
\bigcap_{n=1}^{\infty} G_n$ and $\Sigma_{0,*}(\HHh) = \BBb(\HHh) \setminus G_1$ which clearly finishes the proof.
\end{proof}

Now let $\CCc\RRr(\HHh)$ be the set of all closed range operators of $\BBb(\HHh)$. Our next purpose is to show that
$\CCc\RRr(\HHh)$ is WOT-Borel (and hence Borel as well). This is however not so simple. Firstly, the maps
$\BBb(\HHh) \ni A \mapsto A^* \in \BBb(\HHh)$ and $\BBb(\HHh) \times \BBb(\HHh) \ni (A,B) \mapsto AB \in \BBb(\HHh)$
are WOT-Borel (cf. \cite{effros}), since the first of them is WOT-continuous and the second is SOT-continuous
on bounded sets. Secondly, closed range operators may be characterized in the space $\BBb(\HHh)$ by means
of the so-called \textit{Moore-Penrose inverse} in the following way:

\begin{pro}{penrose}
An operator $A \in \BBb(\HHh)$ has closed range iff there is $B \in \BBb(\HHh)$ such that $ABA = A$, $BAB = B$
and $AB$ and $BA$ are selfadjoint. What is more, the operator $B$ is uniquely determined by these properties and if
$A$ has closed range, $B = A^{\dag} := (A\bigr|_{\RrR(A^*)})^{-1} P$ where $P$ is the orthogonal projection onto
the range of $A$.
\end{pro}

The operator $A^{\dag}$ appearing in the statement of \PRO{penrose} is the above mentioned Moore-Penrose inverse
of an operator $A \in \CCc\RRr(\HHh)$. For the proof of \PRO{penrose}, see e.g. \cite{penrose}.\par
In the proof of the next result we use the theorem on Souslin sets which states that if $Y$ and $Z$ are separable complete
metric spaces, $B$ is a Borel subset of $Y$ and $f\dd B \to Z$ is a continuous one-to-one function, then $f(B)$ is a Borel
subset of $Z$ (see \cite[Theorem~XIII.1.9]{k-m} or \cite[Corollary~A.7]{takesaki}, or \cite[Theorem~A.25]{takesaki3}
for more general result).

\begin{thm}{CR}
The set $\CCc\RRr(\HHh)$ is WOT-Borel.
\end{thm}
\begin{proof}
For $r > 0$ let $B_r$ be the closed ball in $\BBb(\HHh)$ with center at $0$ and of radius $r$ equipped with the weak
operator topology. Let $\psi_r\dd B_r^2 \ni (T,S) \mapsto (TST - T, STS - S, S^*T^* - TS, T^*S^* - ST) \in B_s^4$
where $s := 2(r + 1)^3$. By the note preceding the statement of the theorem, $\psi_r$ is WOT-Borel. So, the set
$C_r := \psi_r^{-1}(\{(0,0,0,0)\})$ is WOT-Borel in $B_r$ as well. By \PRO{penrose}, $C_r = \{(A,A^{\dag})\dd\
A \in \CCc\RRr(\HHh),\ A, A^{\dag} \in B_r\}$. So, the projection $p_r$ of $C_r$ onto the first factor is one-to-one.
But $p_r$ is WOT-continuous and $B_r$ and $B_s$ are compact metrizable spaces. We infer from this that $E_r := p_r(C_r)$
is WOT-Borel in $B_s$. Finally, the observation that $B_s$ is WOT-Borel in $\BBb(\HHh)$ and $\CCc\RRr(\HHh) =
\bigcup_{n=1}^{\infty} E_n$ finishes the proof.
\end{proof}

\textbf{Question 2.} Which additive or multiplicative class, in the hierarchy of WOT-Borel sets in $\BBb(\HHh)$, the set
$\CCc\RRr(\HHh)$ is of?\\
\nl
Now we are ready to prove the following

\begin{thm}{type}
For each $k,n,m=0,1,2,\ldots,\infty$ with $k$ finite:
\begin{enumerate}[\upshape(a)]
\item $\Sigma_k(\HHh)$ is $\FFf \cap \GGg$,
\item $\Sigma_{n,m}^1(\HHh)$ is $\FFf \cap \GGg$ (respectively open) provided $m$ or $n$ is finite (respectively
   $m=0$ or $n=0$),
\item $\Sigma_{0,0}^0(\HHh)$ is $\GGg_{\delta}$; $\Sigma_{n,m}^0(\HHh)$ is $\FFf_{\sigma} \cap \GGg_{\delta}$ for finite
   $n$ and $m$; $\Sigma_{n,m}^0(\HHh)$ is $\FFf_{\sigma\delta}$ if either $n$ or $m$ is infinite (and the other is finite),
\item $\Sigma_{\infty,\infty}^1(\HHh)$ and $\Sigma_{\infty,\infty}^0(\HHh)$ are Borel.
\end{enumerate}
\end{thm}
\begin{proof}
Clearly, for each finite $k$ the set of all finite rank operators $A \in \BBb(\HHh)$ such that $\dim \RrR(A) \leqsl k$
is closed and therefore $\Sigma_k(\HHh)$ is $\FFf \cap \GGg$. Further, for each $n=0,1,2\ldots,\infty$ let
$\Sigma_{n,*}(\HHh)$ be as in \LEM{ker=n} and let $\Sigma_{*,n}(\HHh) = \{A^*\dd\ A \in \Sigma_{n,*}(\HHh)\}$.
Observe that $\Sigma_{*,n}(\HHh)$ is of the same Borel class as $\Sigma_{n,*}(\HHh)$.\par
Suppose that $n$ or $m$ is finite. With no loss of generality, we may assume that $n \leqsl m$. Put $k = n - m \in
\ZZZ \cup \{-\infty\}$. The set $F(k)$ of all semi-Fredholm operators of index $k$ is open in $\BBb(\HHh)$ (see e.g.
Proposition XI.2.4 and Theorem XI.3.2 in \cite{conway}). It may also be shown that for each integer $l \geqsl 0$ the set
$F_l(k)$ of all $A \in F(k)$ such that $\dim \NnN(A) \leqsl l$ is open in $\BBb(\HHh)$ as well (see e.g.
\cite[Proposition~5.3]{pn}). Now the relation $\Sigma_{n,m}^1(\HHh) = F_n(k) \setminus F_{n-1}(k)$ (with $F_{-1}(k) =
\varempty$) shows (b). Further, since $\Sigma_{n,m}^0(\HHh) = \Sigma_{n,*}(\HHh) \cap \Sigma_{*,m}(\HHh) \setminus F(k)$,
we infer from \LEM{ker=n} the assertion of (c).\par
Finally, $\Sigma_{\infty,\infty}^1(\HHh) = \Sigma_{\infty,*}(\HHh) \cap \Sigma_{*,\infty}(\HHh) \cap \CCc\RRr(\HHh)
\setminus \bigcup_{n=0}^{\infty} \Sigma_n(\HHh)$. So, this set is Borel by \THM{CR}. Since $\Sigma_{\infty,\infty}^0(\HHh)$
is the complement in $\BBb(\HHh)$ of the union of all other members of $\Sigma(\HHh)$, it is Borel as well.
\end{proof}

\textbf{Question 3.} Which additive or multiplicative class, in the hierarchy of Borel sets in $\BBb(\HHh)$, the sets
$\Sigma_{\infty,\infty}^1(\HHh)$ and $\Sigma_{\infty,\infty}^0(\HHh)$ are of?\\
\nl
Now let $T$ be a closed densely defined operator in $\HHh$. By $\sigma(T)$ we denote the spectrum of $T$. That is,
a complex number $\lambda$ does \textbf{not} belong to $\sigma(T)$ iff $\NnN(T - \lambda) = \{0\}$, $\RrR(T - \lambda) =
\HHh$ and $(T - \lambda)^{-1}$ is bounded (the latter condition may be omitted by the Closed Graph Theorem). We decompose
the complex plane into parts corresponding to the members of $\Sigma(\HHh)$:
\begin{itemize}
\item $\sigma^f(T)$ is the set of all $z \in \CCC$ such that $\RrR(T - z)$ is finite-dimensional,
\item for $n,m=0,1,2,\ldots,\infty$ let $\sigma_{n,m}^1(T)$ and $\sigma_{n,m}^0(T)$ be the sets consisting of all
   $z \in \CCC$ for which $\dim \NnN(T - z) = n$, $\dim \RrR(T - z)^{\perp} = m$, $\dim \overline{\RrR}(T - z) = \infty$
   and, respectively, $\RrR(T - z)$ is closed or not.
\end{itemize}
Notice that $\CCC \setminus \sigma(T) = \sigma_{0,0}^1(T)$ is the resolvent set of $T$ and that the above defined sets are
pairwise disjoint and cover the complex plane. The collection of all of them is denoted by $\Sigma(T)$. We say that
the sets $\sigma_{n,m}^1(T)$ and $\sigma_{n,m}^0(T)$ correspond to, respectively, $\Sigma_{n,m}^1(\HHh)$
and $\Sigma_{n,m}^0(\HHh)$. What we want is to prove

\begin{pro}{plane}
For every closed densely defined operator $T$ in $\HHh$, $\Sigma(T)$ consists of Borel subsets of $\CCC$. What is more,
$\card \sigma^f(T) \leqsl 1$ and each member of $\Sigma(T)$ different from $\sigma^f(T)$ is of the same Borel class
as the corresponding member of $\Sigma(\HHh)$.
\end{pro}
\begin{proof}
First observe that if $z \in \sigma^f(T)$, then $\DdD(T) = \HHh$ and $T$ is bounded (since $\NnN(T - z)$ is closed).
We conclude from this that indeed $\card \sigma^f(T) \leqsl 1$. Further, since $T$ is closed, there is an operator
$A \in \BBb(\HHh)$ such that $\NnN(A) = \{0\}$ and $\RrR(A) = \DdD(T)$ (e.g. $A := (|T| + 1)^{-1}$; see also
\cite[Theorem~1.1]{ranges}). Put $C = TA \in \BBb(\HHh)$. Notice that for each $z \in \CCC$, $\RrR(C - z A) = \RrR(T - z)$
and $\NnN(C - z A) = A^{-1}(\NnN(T - z))$, and thus $\dim \NnN(T - z) = \dim \NnN(C - z A)$. It is infered from this
that $z \in \sigma_{n,m}^j(T)$ iff $C - zA \in \Sigma_{n,m}^j(\HHh)$. So, the continuity of the function $\CCC \ni z
\mapsto C - zA \in \BBb(\HHh)$ finishes the proof.
\end{proof}

We end the paper with notes concerning the so-called \textit{residual} and \textit{continuous} spectra of an operator.
In the literature there are at least two different, nonequivalent definitions of them. For some mathematicians the residual
spectrum $\sigma_r(T)$ of a closed densely defined operator $T$ in a separable Hilbert space $\HHh$ consists of all
$\lambda \in \sigma(T) \setminus \sigma_p(T)$ such that $\RrR(T - \lambda)$ is closed (e.g. \cite{b-s}), for others it is
the set of all $\lambda \in \sigma(T) \setminus \sigma_p(T)$ such that $\overline{\RrR}(T - \lambda) \neq \HHh$
(e.g. \cite{mlak}). Similarly, the continuous spectrum $\sigma_c(T)$ of $T$ for some people means the set of all $\lambda
\in \KKK$ such that $\RrR(T - \lambda)$ is nonclosed (e.g. \cite{b-s}), for others it coincides with $\sigma(T)
\setminus (\sigma_p(T) \cup \sigma_r(T))$ (e.g. \cite{mlak}). Somehow or other --- whatever definitions we choose among
the above, both the residual and the continuous spectra are Borel and are the unions of subfamilies of $\Sigma(T)$.

\end{document}

%% file: pream.tex
\usepackage{latexsym}
\usepackage{ifthen}
\usepackage[leqno]{amsmath}
\usepackage{enumerate}
\usepackage{calc}
\usepackage{amstext,amsbsy,amsopn,amsthm,amsgen,amsfonts,amscd,amsxtra,upref}
\usepackage{mathrsfs}\usepackage{euscript}\usepackage{amssymb}

\swapnumbers
\theoremstyle{plain}
\newtheorem{Thm}{Theorem}[section]
\newtheorem{Lem}[Thm]{Lemma}
\newtheorem{Cor}[Thm]{Corollary}
\newtheorem{Pro}[Thm]{Proposition}
\newtheorem{Prp}[Thm]{Properties}
\newtheorem{Sub}[Thm]{Sublemma}

\theoremstyle{definition}
\newtheorem{Def}[Thm]{Definition}
\newtheorem{Exm}[Thm]{Example}
\newtheorem{Exs}[Thm]{Examples}

\theoremstyle{remark}
\newtheorem{Rem}[Thm]{Remark}
\newtheorem{Rms}[Thm]{Remarks}
\newtheorem*{Com}{Commentary}


%% file: defin.tex
\newcommand{\myEmail}{piotr.niemiec@uj.edu.pl}
\newcommand{\myAddress}{\noindent{}Piotr Niemiec\\{}Jagiellonian University\\{}Institute of Mathematics\\{}
   ul. \L{}ojasiewicza 6\\{}30-348 Krak\'{o}w\\{}Poland}
\newcommand{\myData}{\author[P. Niemiec]{Piotr Niemiec}\address{\myAddress}\email{\myEmail}}

\newcommand{\CCC}{\mathbb{C}}\newcommand{\DDD}{\mathbb{D}}

\newcommand{\KKK}{\mathbb{K}}
\newcommand{\NNN}{\mathbb{N}}
\newcommand{\RRR}{\mathbb{R}}

\newcommand{\ZZZ}{\mathbb{Z}}
\newcommand{\BBb}{\CMcal{B}}\newcommand{\CCc}{\CMcal{C}}\newcommand{\DDd}{\CMcal{D}}
\newcommand{\EEe}{\CMcal{E}}\newcommand{\FFf}{\CMcal{F}}\newcommand{\GGg}{\CMcal{G}}\newcommand{\HHh}{\CMcal{H}}

\newcommand{\RRr}{\CMcal{R}}
\newcommand{\WWw}{\CMcal{W}}

\newcommand{\DdD}{\EuScript{D}}

\newcommand{\NnN}{\EuScript{N}}
\newcommand{\RrR}{\EuScript{R}}

\newcommand{\XxX}{\EuScript{X}}

\newcommand{\mM}{\mathfrak{m}}

\newcommand{\SECT}[1]{\section{#1}\renewcommand{\theequation}{\thesection-\arabic{equation}}\setcounter{equation}{0}}

\newcounter{help}
\newcommand{\ITE}[3]{\ifthenelse{#1}{#2}{#3}}\newcommand{\ITEE}[3]{\ITE{\equal{#1}{#2}}{#3}{}}

\newcommand{\card}{\operatorname{card}}

\newcommand{\leqsl}{\leqslant}\newcommand{\geqsl}{\geqslant}

\newcommand{\epsi}{\varepsilon}\newcommand{\varempty}{\varnothing}\newcommand{\dd}{\colon}

\newcommand{\nl}{${}$\newline}

\newcommand{\THM}[1]{Theorem~\textup{\ref{thm:#1}}}
\newcommand{\COR}[1]{Corollary~\textup{\ref{cor:#1}}}
\newcommand{\LEM}[1]{Lemma~\textup{\ref{lem:#1}}}\newcommand{\PRO}[1]{Proposition~\textup{\ref{pro:#1}}}

\newcommand{\EXM}[1]{Example~\textup{\ref{exm:#1}}}

\newenvironment{thm}[1]{\begin{Thm}\label{thm:#1}}{\end{Thm}}\newenvironment{lem}[1]{\begin{Lem}\label{lem:#1}}{\end{Lem}}
\newenvironment{cor}[1]{\begin{Cor}\label{cor:#1}}{\end{Cor}}\newenvironment{pro}[1]{\begin{Pro}\label{pro:#1}}{\end{Pro}}

\newenvironment{exm}[1]{\begin{Exm}\label{exm:#1}}{\end{Exm}}
\newenvironment{rem}[1]{\begin{Rem}\label{rem:#1}}{\end{Rem}}

\input{bibl}


%% file: bibl.tex
\newcommand{\bibITEM}[2]{\ITE{\equal{#2}{}}{\bibitem{#1} }{\bibitem[#2]{#1} }}
\newcommand{\BIB}[8]{
   \bibITEM{#1}{#8} #2, \textit{#3}, #4{} \textbf{#5} (#6), #7.}
\newcommand{\myBIB}[6][P. Niemiec]{#1, \textit{#2}, #3{}\ITE{\equal{#4}{}}{}{ \textbf{#4}} (#5), #6.}
\newcommand{\BIb}[6]{
   \bibITEM{#1}{#6} #2, \textit{#3}, #4, #5.}
\newcommand{\BiB}[9]{
   \bibITEM{#1}{#9} #2, \textit{#3}, #4{} \textit{#5}, #6, #7, #8.}

\newcommand{\myBAPP}[3][P. Niemiec]{
   #1, \textit{#2}, #3}

\newcommand{\jRN}[2][]{
   \ITEE{#2}{ActaM}{\ITE{\equal{#1}{+}}
      {Acta Mathematica}{Acta Math.}}
   \ITEE{#2}{ActaMSinES}{\ITE{\equal{#1}{+}}
      {Acta Mathematica Sinica (English Series)}{Acta Math. Sin. (Engl. Ser.)}}
   \ITEE{#2}{AdvM}{\ITE{\equal{#1}{+}}
      {Advances in Mathematics}{Adv. in Math.}}
   \ITEE{#2}{ACS}{\ITE{\equal{#1}{+}}
      {Applied Categorical Structures}{Appl. Categor. Struct.}}
   \ITEE{#2}{ActaSM}{\ITE{\equal{#1}{+}}
      {Acta Scientiarum Mathematicarum}{Acta Sci. Math.}}
   \ITEE{#2}{AmJM}{\ITE{\equal{#1}{+}}
      {American Journal of Mathematics}{Amer. J. Math.}}
   \ITEE{#2}{AmMMon}{\ITE{\equal{#1}{+}}
      {American Mathematical Monthly}{Amer. Math. Mon.}}
   \ITEE{#2}{AnnSciEcNormSupT}{\ITE{\equal{#1}{+}}
      {Annales Scientifiques de l'\'{E}cole Normale Sup\'{e}rieure (3)}{Ann. Sci. \'{E}c. Norm. Sup\'{e}r. (3)}}
   \ITEE{#2}{AnnM}{\ITE{\equal{#1}{+}}
      {Annals of Mathematics}{Ann. Math.}}
   \ITEE{#2}{AnnProb}{\ITE{\equal{#1}{+}}
      {The Annals of Probability}{Ann. Probab.}}
   \ITEE{#2}{AnnPALog}{\ITE{\equal{#1}{+}}
      {Annals of Pure and Applied Logic}{Ann. Pure Appl. Logic}}
   \ITEE{#2}{APM}{\ITE{\equal{#1}{+}}
      {Annales Polonici Mathematici}{Ann. Polon. Math.}}
   \ITEE{#2}{ArchM}{\ITE{\equal{#1}{+}}
      {Archiv der Mathematik}{Arch. Math.}}
   \ITEE{#2}{AttiAccLincRendNat}{\ITE{\equal{#1}{+}}
      {Atti della Accademia Nazionale dei Lincei. Rendiconti. Classe di Scienze Fisiche, Matematiche e Naturali}
      {Atti Accad. Naz. Lincei Rend. Cl. Sci. Fis. Mat. Nat.}}
   \ITEE{#2}{BAMS}{\ITE{\equal{#1}{+}}
      {Bulletin of the American Mathematical Society}{Bull. Amer. Math. Soc.}}
   \ITEE{#2}{BAustrMS}{\ITE{\equal{#1}{+}}
      {Bulletin of the Australian Mathematical Society}{Bull. Austral. Math. Soc.}}
   \ITEE{#2}{BLondMS}{\ITE{\equal{#1}{+}}
      {Bulletin of the London Mathematical Sociecy}{Bull. Lond. Math. Soc.}}
   \ITEE{#2}{BAPolSSSM}{\ITE{\equal{#1}{+}}
      {Bulletin de l'Acad\'{e}mie Polonaise des Sciences. S\'{e}rie des Sciences Math\'{e}matiques}
      {Bull. Acad. Pol. Sci. S\'{e}r. Sci. Math.}}
   \ITEE{#2}{BullSM}{\ITE{\equal{#1}{+}}
      {Bulletin des Sciences Math\'{e}matiques}{Bull. Sci. Math.}}
   \ITEE{#2}{BullPol}{\ITE{\equal{#1}{+}}
      {Bulletin of the Polish Academy of Sciences: Mathematics}{Bull. Pol. Acad. Sci. Math.}}
   \ITEE{#2}{CanadJM}{\ITE{\equal{#1}{+}}
      {Canadian Journal Mathematics}{Canad. J. Math.}}
   \ITEE{#2}{CollectM}{\ITE{\equal{#1}{+}}
      {Collectanea Mathematica}{Collect. Math.}}
   \ITEE{#2}{CMUC}{\ITE{\equal{#1}{+}}
      {Commentationes Mathematicae Universitatis Carolinae}{Comment. Math. Univ. Carolin.}}
   \ITEE{#2}{CRParis}{\ITE{\equal{#1}{+}}
      {C. R. Paris}{C. R. Paris}}
   \ITEE{#2}{CRASParis}{\ITE{\equal{#1}{+}}
      {Comptes Rendus de l'Acad\'{e}mie des Sciences. Paris}{C. R. Acad. Sci. Paris}}
   \ITEE{#2}{CEurJM}{\ITE{\equal{#1}{+}}
      {Central European Journal of Mathematics}{Cent. Eur. J. Math.}}
   \ITEE{#2}{CMHelv}{\ITE{\equal{#1}{+}}
      {Commentarii Mathematici Helvetici}{Comment. Math. Helv.}}
   \ITEE{#2}{CollM}{\ITE{\equal{#1}{+}}
      {Colloquium Mathematicum}{Coll. Math.}}
   \ITEE{#2}{ComposM}{\ITE{\equal{#1}{+}}
      {Compositio Mathematica}{Compos. Math.}}
   \ITEE{#2}{CzMJ}{\ITE{\equal{#1}{+}}
      {Czechoslovak Mathematical Journal}{Czech. Math. J.}}
   \ITEE{#2}{DissM}{\ITE{\equal{#1}{+}}
      {Dissertationes Mathematicae}{Dissert. Math.}}
   \ITEE{#2}{DANSSSR}{\ITE{\equal{#1}{+}}
      {Doklady Akademii Nauk SSSR}{Dokl. Akad. Nauk SSSR}}
   \ITEE{#2}{DukeMJ}{\ITE{\equal{#1}{+}}
      {Duke Mathematical Journal}{Duke Math. J.}}
   \ITEE{#2}{ELA}{\ITE{\equal{#1}{+}}
      {The Electronic Journal of Linear Algebra}{Electron. J. Linear Algebra}}
   \ITEE{#2}{ExtrM}{\ITE{\equal{#1}{+}}
      {Extracta Mathematicae}{Extracta Math.}}
   \ITEE{#2}{FM}{\ITE{\equal{#1}{+}}
      {Fundamenta Mathematicae}{Fund. Math.}}
   \ITEE{#2}{FAA}{\ITE{\equal{#1}{+}}
      {Functional Analysis and its Applications}{Funct. Anal. Appl.}}
   \ITEE{#2}{FunkAnalPril}{\ITE{\equal{#1}{+}}
      {Funktsional'ny\u{\i} Analiz i Ego Prilozheniya}{Funkts. Anal. Prilozh.}}
   \ITEE{#2}{GTopA}{\ITE{\equal{#1}{+}}
      {General Topology and its Applications}{General Topol. Appl.}}
   \ITEE{#2}{HJM}{\ITE{\equal{#1}{+}}
      {Houston Journal of Mathematics}{Houston J. Math.}}
   \ITEE{#2}{IllinoisJM}{\ITE{\equal{#1}{+}}
      {Illinois Journal of Mathematics}{Illinois J. Math.}}
   \ITEE{#2}{IndagMP}{\ITE{\equal{#1}{+}}
      {Indagationes Mathematicae (Proceedings)}{Indagationes Math. Proc.}}
   \ITEE{#2}{IndianaUMJ}{\ITE{\equal{#1}{+}}
      {Indiana University Mathematical Journal}{Indiana Univ. Math. J.}}
   \ITEE{#2}{InHauEtSPM}{\ITE{\equal{#1}{+}}
      {Inst. Hautes \'{E}tudes Sci. Publ. Math.}{Inst. Hautes \'{E}tudes Sci. Publ. Math.}}
   \ITEE{#2}{IEOT}{\ITE{\equal{#1}{+}}
      {Integral Equations and Operator Theory}{Integr. Equ. Oper. Theory}}
   \ITEE{#2}{IsraelJM}{\ITE{\equal{#1}{+}}
      {Israel Journal of Mathematics}{Israel J. Math.}}
   \ITEE{#2}{JAusMSA}{\ITE{\equal{#1}{+}}
      {Journal of the Australian Mathematical Society. Series A}{J. Aust. Math. Soc. Ser. A}}
   \ITEE{#2}{JCA}{\ITE{\equal{#1}{+}}
      {Journal of Convex Analysis}{J. Convex Anal.}}
   \ITEE{#2}{JChinUST}{\ITE{\equal{#1}{+}}
      {J. China Univ. Sci. Tech.}{J. China Univ. Sci. Tech.}}
   \ITEE{#2}{JFA}{\ITE{\equal{#1}{+}}
      {Journal of Functional Analysis}{J. Funct. Anal.}}
   \ITEE{#2}{JKoreanMS}{\ITE{\equal{#1}{+}}
      {Journal of the Korean Mathematical Society}{J. Korean Math. Soc.}}
   \ITEE{#2}{JMAnApp}{\ITE{\equal{#1}{+}}
      {J. Math. Anal. Appl.}{J. Math. Anal. Appl.}}
   \ITEE{#2}{JOT}{\ITE{\equal{#1}{+}}
      {Journal of Operator Theory}{J. Operator Theory}}
   \ITEE{#2}{KodaiMSemRep}{\ITE{\equal{#1}{+}}
      {Kodai Math. Sem. Rep.}{Kodai Math. Sem. Rep.}}
   \ITEE{#2}{LAA}{\ITE{\equal{#1}{+}}
      {Linear Algebra and its Applications}{Linear Algebra Appl.}}
   \ITEE{#2}{LMLA}{\ITE{\equal{#1}{+}}
      {Linear and Multilinear Algebra}{Linear Multilinear Algebra}}
   \ITEE{#2}{LNM}{\ITE{\equal{#1}{+}}
      {Lecture Notes in Mathematics}{Lecture Notes Math.}}
   \ITEE{#2}{MathJap}{\ITE{\equal{#1}{+}}
      {Math. Japon.}{Math. Japon.}}
   \ITEE{#2}{MLQ}{\ITE{\equal{#1}{+}}
      {Mathematical Logic Quarterly}{Math. Log. Q.}}
   \ITEE{#2}{MProcCambPhS}{\ITE{\equal{#1}{+}}
      {Mathematical Proceedings of the Cambridge Philosophical Society}{Math. Proc. Cambridge Phil. Soc.}}
   \ITEE{#2}{MMag}{\ITE{\equal{#1}{+}}
      {Mathematics Magazine}{Math. Mag.}}
   \ITEE{#2}{MSb}{\ITE{\equal{#1}{+}}
      {Matematicheski\u{\i} Sbornik}{Mat. Sb.}}
   \ITEE{#2}{MStud}{\ITE{\equal{#1}{+}}
      {Matematychni Studi\"{\i}}{Mat. Stud.}}
   \ITEE{#2}{MScand}{\ITE{\equal{#1}{+}}
      {Mathematica Scandinavica}{Math. Scand.}}
   \ITEE{#2}{MAnn}{\ITE{\equal{#1}{+}}
      {Mathematische Annalen}{Math. Ann.}}
   \ITEE{#2}{MAMS}{\ITE{\equal{#1}{+}}
      {Memoirs of the American Mathematical Society}{Mem. Amer. Math. Soc.}}
   \ITEE{#2}{MichMJ}{\ITE{\equal{#1}{+}}
      {Michigan Mathematical Journal}{Mich. Math. J.}}
   \ITEE{#2}{MonatM}{\ITE{\equal{#1}{+}}
      {Monatshefte f\"{u}r Mathematik}{Mh. Math.}}
   \ITEE{#2}{NonlinA}{\ITE{\equal{#1}{+}}
      {Nonlinear Analysis: Theory, Methods \& Applications}{Nonlinear Anal.}}
   \ITEE{#2}{NAMS}{\ITE{\equal{#1}{+}}
      {Notices of the American Mathematical Society}{Notices Amer. Math. Soc.}}
   \ITEE{#2}{OpusM}{\ITE{\equal{#1}{+}}
      {Opuscula Mathematica}{Opuscula Math.}}
   \ITEE{#2}{PacJM}{\ITE{\equal{#1}{+}}
      {Pacific Journal of Mathematics}{Pacific J. Math.}}
   \ITEE{#2}{PeriodMHung}{\ITE{\equal{#1}{+}}
      {Periodica Mathematica Hungarica}{Period. Math. Hungarica}}
   \ITEE{#2}{PAMS}{\ITE{\equal{#1}{+}}
      {Proceedings of the American Mathematical Society}{Proc. Amer. Math. Soc.}}
   \ITEE{#2}{ProcCambPhS}{\ITE{\equal{#1}{+}}
      {Proceedings of the Cambridge Philosophical Society}{Proc. Cambridge Phil. Soc.}}
   \ITEE{#2}{ProcImpAcadTokyo}{\ITE{\equal{#1}{+}}
      {Proc. Imp. Acad. Tokyo}{Proc. Imp. Acad. Tokyo}}
   \ITEE{#2}{ProcKonink}{\ITE{\equal{#1}{+}}
      {Proceedings of the Koninklijke Nederlandse Akademie van Wetenschappen}{Nederl. Akad. Wetensch. Proc. Ser. A}}
   \ITEE{#2}{PLondMS}{\ITE{\equal{#1}{+}}
      {Proceedings of the London Mathematical Society}{Proc. London Math. Soc.}}
   \ITEE{#2}{PNatlUSA}{\ITE{\equal{#1}{+}}
      {Proceedings of the National Academy of Sciences of the United States of America}{Proc. Natl. Acad. Sci. USA}}
   \ITEE{#2}{PublRIMSKyoto}{\ITE{\equal{#1}{+}}
      {Publ. Res. Inst. Math. Sci. Kyoto Univ.}{Publ. Res. Inst. Math. Sci.}}
   \ITEE{#2}{PWN}{\ITE{\equal{#1}{+}}
      {PWN -- Polish Scientific Publishers, Warszawa}{PWN -- Polish Scientific Publishers, Warszawa}}
   \ITEE{#2}{RCMP}{\ITE{\equal{#1}{+}}
      {Rendiconti del Circolo Matematico di Palermo}{Rend. Circ. Mat. Palermo}}
   \ITEE{#2}{RussMS}{\ITE{\equal{#1}{+}}
      {Russian Mathematical Surveys}{Russian Math. Surveys}}
   \ITEE{#2}{SbM}{\ITE{\equal{#1}{+}}
      {Sbornik: Mathematics}{Sb. Math.}}
   \ITEE{#2}{SciRepTokyoA}{\ITE{\equal{#1}{+}}
      {Science Reports of Tokyo Kyoiku Daigaku, Section A}{Sci. Rep. Tokyo Kyoiku Daigaku Sect. A}}
   \ITEE{#2}{SeminProbStras}{\ITE{\equal{#1}{+}}
      {S\'{e}minaire de probabilit\'{e}s de Strasbourg}{S\'{e}min. Prob. Strasbourg}}
   \ITEE{#2}{SIAMJMAA}{\ITE{\equal{#1}{+}}
      {SIAM Journal on Matrix Analysis and Applications}{SIAM J. Matrix Anal. Appl.}}
   \ITEE{#2}{SibirMZ}{\ITE{\equal{#1}{+}}
      {Sibirski\v{\i} Mat. \v{Z}hurnal}{Sibirsk. Mat. \v{Z}.}}
   \ITEE{#2}{SM}{\ITE{\equal{#1}{+}}
      {Studia Mathematica}{Studia Math.}}
   \ITEE{#2}{TAMS}{\ITE{\equal{#1}{+}}
      {Transactions of the American Mathematical Society}{Trans. Amer. Math. Soc.}}
   \ITEE{#2}{TohokuMJ}{\ITE{\equal{#1}{+}}
      {T\^{o}hoku Mathematical Journal}{T\^{o}hoku Math. J.}}
   \ITEE{#2}{TomskUnivRev}{\ITE{\equal{#1}{+}}
      {Tomsk Universitet Review}{Tomsk. Univ. Rev.}}
   \ITEE{#2}{TopA}{\ITE{\equal{#1}{+}}
      {Topology and its Applications}{Topology Appl.}}
   \ITEE{#2}{TopMethNA}{\ITE{\equal{#1}{+}}
      {Topological Methods in Nonlinear Analysis}{Topol. Methods Nonlinear Anal.}}
   \ITEE{#2}{TsukubaJM}{\ITE{\equal{#1}{+}}
      {Tsukuba Journal of Mathematics}{Tsukuba J. Math.}}
   \ITEE{#2}{UspekhiMN}{\ITE{\equal{#1}{+}}
      {Uspekhi Matem. Nauk}{Uspekhi Mat. Nauk}}
   }

\newcommand{\paplist}[3][]{
   \ITEE{#3}{NIAkhiezer,IMGlazman1993}{
      \BIb{#2}{N.I. Akhiezer and I.M. Glazman}
         {Theory of Linear Operators in Hilbert Space}
         {Dover Publications, Inc., New York}{1993}{#1}}
   \ITEE{#3}{RDAnderson1967}{
      \BIB{#2}{R.D. Anderson}
         {On topological infinite deficiency}
         {\jRN{MichMJ}}{14}{1967}{365--383}{#1}}
   \ITEE{#3}{RDAnderson,JMcCharen1970}{
      \BIB{#2}{R.D. Anderson and J. McCharen}
         {On extending homeomorphisms to Fr\'{e}chet manifolds}
         {\jRN{PAMS}}{25}{1970}{283--289}{#1}}
   \ITEE{#3}{RDAnderson,DWCurtis,JVanMill1982}{
      \BIB{#2}{R.D. Anderson, D.W. Curtis, J. van Mill}
         {A fake topological Hilbert space}
         {\jRN{TAMS}}{272}{1982}{311--321}{#1}}
   \ITEE{#3}{RArens,JEells1956}{
      \BIB{#2}{R. Arens and J. Eells}
         {On embedding uniform and topological spaces}
         {\jRN{PacJM}}{6}{1956}{397--403}{#1}}
   \ITEE{#3}{NAronszajn,PPanitchpakdi1956}{
      \BIB{#2}{N. Aronszajn and P. Panitchpakdi}
         {Extension of uniformly continuous transformations and hyperconvex metric spaces}
         {\jRN{PacJM}}{6}{1956}{405--439}{#1}}
   \ITEE{#3}{KJBabenko1948}{
      \BIB{#2}{K.J. Babenko}
         {On conjugate functions}
         {\jRN{DANSSSR}}{62}{1948}{157--160}{#1}}
   \ITEE{#3}{TBanakh1995}{
      \BIB{#2}{T.O. Banakh}
         {Topology of spaces of probability measures, I}
         {\jRN{MStud}}{5}{1995}{65--87 (Russian)}{#1}}
   \ITEE{#3}{TBanakh1995a}{
      \BIB{#2}{T.O. Banakh}
         {Topology of spaces of probability measures, II}
         {\jRN{MStud}}{5}{1995}{88--106 (Russian)}{#1}}
   \ITEE{#3}{TBanakh1998}{
      \BIB{#2}{T. Banakh}
         {Characterization of spaces admitting a homotopy dense embedding into a Hilbert manifold}
         {\jRN{TopA}}{86}{1998}{123--131}{#1}}
   \ITEE{#3}{TBanakh,TNRadul1997}{
      \BIB{#2}{T.O. Banakh and T.N. Radul}
         {Topology of spaces of probability measures}
         {\jRN{SbM}}{188}{1997}{973--995}{#1}}
   \ITEE{#3}{TBanakh,TRadul,MZarichnyi1996}{
      \BIb{#2}{T. Banakh, T. Radul, M. Zarichnyi}
         {Absorbing sets in infinite-dimensional manifolds}
         {VNTL Publishers, Lviv}{1996}{#1}}
   \ITEE{#3}{TBanakh,IZarichnyy2008}{
      \BIB{#2}{T. Banakh and I. Zarichnyy}
         {Topological groups and convex sets homeomorphic to non-separable Hilbert spaces}
         {\jRN{CEurJM}}{6}{2008}{77--86}{#1}}
   \ITEE{#3}{HBecker,ASKechris1996}{
      \BIb{#2}{H. Becker and A.S. Kechris}
         {The Descriptive Set Theory of Polish Group Actions \textup{(London Math. Soc. Lecture Note Series, vol. 232)}}
         {University Press, Cambridge}{1996}{#1}}
   \ITEE{#3}{GBeer1993}{
      \BIb{#2}{G. Beer}
         {Topologies on Closed and Closed Convex Sets \textup{(Mathematics and Its Applications)}}
         {Kluwer Academic Publishers, Dordrecht}{1993}{#1}}
   \ITEE{#3}{NEBenamara,NNikolski1999}{
      \BIB{#2}{N.E. Benamara and N. Nikolski}
         {Resolvent tests for similarity to a normal operator}
         {\jRN{PLondMS}}{78}{1999}{585--626}{#1}}
   \ITEE{#3}{YBenyamini,JLindenstrauss2000}{
      \BIb{#2}{Y. Benyamini and J. Lindenstrauss}
         {Geometric nonlinear functional analysis I}
         {AMS Colloquium Publications 48}{2000}{#1}}
   \ITEE{#3}{SKBerberian1974}{
      \BIb{#2}{S.K. Berberian}
         {Lectures in Functional Analysis and Operator Theory}
         {Graduate Texts in Mathematics 15, Springer-Verlag, New York}{1974}{#1}}
   \ITEE{#3}{SNBernstein1954}{
      \BIb{#2}{S.N. Bernstein}
         {Collected Works II}
         {Akad. Nauk SSSR, Moscow}{1954 (Russian)}{#1}}
   \ITEE{#3}{CzBessaga,APelczynski1972}{
      \BIB{#2}{Cz. Bessaga and A. Pe\l{}czy\'{n}ski}
         {On spaces of measurable functions}
         {\jRN{SM}}{44}{1972}{597--615}{#1}}
   \ITEE{#3}{CzBessaga,APelczynski1975}{
      \BIb{#2}{Cz. Bessaga and A. Pe\l{}czy\'{n}ski}
         {Selected topics in infinite-dimensional topology}
         {\jRN{PWN}}{1975}{#1}}
   \ITEE{#3}{MBestvina,JMogilski1986}{
      \BIB{#2}{M. Bestvina and J. Mogilski}
         {Characterizing certain incomplete infinite-dimensional absolute retracts}
         {\jRN{MichMJ}}{33}{1986}{291--313}{#1}}
   \ITEE{#3}{MBestvina,PBowers,JMogilsky,JWalsh1986}{
      \BIB{#2}{M. Bestvina, P. Bowers, J. Mogilsky, J. Walsh}
         {Characterization of Hilbert space manifolds revisited}
         {\jRN{TopA}}{24}{1986}{53--69}{#1}}
   \ITEE{#3}{RBhatia1997}{
      \BIb{#2}{R. Bhatia}
         {Matrix Analysis}
         {Springer, New York}{1997}{#1}}
   \ITEE{#3}{GBirkhoff1936}{
      \BIB{#2}{G. Birkhoff}
         {A note on topological groups}
         {\jRN{ComposM}}{3}{1936}{427--430}{#1}}
   \ITEE{#3}{MSBirman,MZSolomjak1987}{
      \BIb{#2}{M.S. Birman and M.Z. Solomjak}
         {Spectral Theory of Self-Adjoint Operators in Hilbert Space}
         {D. Reidel Publishing Co., Dordrecht}{1987}{#1}}
   \ITEE{#3}{EBishop1961}{
      \BIB{#2}{E. Bishop}
         {A generalization of the Stone-Weierstrass theorem}
         {\jRN{PacJM}}{11}{1961}{777--783}{#1}}
   \ITEE{#3}{BBlackadar2006}{
      \BIb{#2}{B. Blackadar}{Operator Algebras. Theory of $\CCc^*$-algebras and von Neumann algebras 
         \textup{(Encyclopaedia of Mathematical Sciences, vol. 122: Operator Algebras and Non-Commutative Geometry III)}}
         {Springer-Verlag, Berlin-Heidelberg}{2006}{#1}}
   \ITEE{#3}{JBlass,WHolsztynski1972}{
      \BIB{#2}{J. Blass and W. Holszty\'{n}ski}
         {Cubical polyhedra and homotopy III}
         {\jRN{AttiAccLincRendNat}}{53}{1972}{275--279}{#1}}
   \ITEE{#3}{FFBonsall,NJDuncan1973}{
      \BIb{#2}{F.F. Bonsall and N.J. Duncan}
         {Complete Normed Algebras}
         {Springer Verlag, Berlin}{1973}{#1}}
   \ITEE{#3}{NBourbaki2002}{
      \BIb{#2}{N. Bourbaki}
         {Lie Groups and Lie Algebras, Chapters 4--6}
         {Springer, New York}{2002}{#1}}
   \ITEE{#3}{PLBowers1989}{
      \BIB{#2}{P.L. Bowers}
         {Limitation topologies on function spaces}
         {\jRN{TAMS}}{314}{1989}{421--431}{#1}}
   \ITEE{#3}{ABrown1953}{
      \BIB{#2}{A. Brown}
         {On a class of operators}
         {\jRN{PAMS}}{4}{1953}{723--728}{#1}}
   \ITEE{#3}{ABrown,CKFong,DWHadwin1978}{
      \BIB{#2}{A. Brown, C.-K. Fong, D.W. Hadwin}
         {Parts of operators on Hilbert space}
         {\jRN{IllinoisJM}}{22}{1978}{306--314}{#1}}
   \ITEE{#3}{AMBruckner,JBBruckner,BSThomson1997}{
      \BIb{#2}{A.M. Bruckner, J.B. Bruckner, B.S. Thomson}
         {Real Analysis}
         {Prentice-Hall, New Jersey}{1997}{#1}}
   \ITEE{#3}{PJCameron,AMVershik2006}{
      \BIB{#2}{P.J. Cameron and A.M. Vershik}
         {Some isometry groups of Urysohn space}
         {\jRN{AnnPALog}}{143}{2006}{70--78}{#1}}
   \ITEE{#3}{CCastaing1966}{
      \BIB{#2}{C. Castaing}
         {Quelques probl\`{e}mes de mesurabilit\'{e} li\'{e}es \`{a} la th\'{e}orie de la commande}
         {\jRN{CRParis}}{262}{1966}{409--411}{#1}}
   \ITEE{#3}{JAVanCasteren1980}{
      \BIB{#2}{J.A. van Casteren}
         {A problem of Sz.-Nagy}
         {\jRN{ActaSM}}{42}{1980}{189--194}{#1}}
   \ITEE{#3}{JAVanCasteren1983}{
      \BIB{#2}{J.A. van Casteren}
         {Operators similar to unitary or selfadjoint ones}
         {\jRN{PacJM}}{104}{1983}{241--255}{#1}}
   \ITEE{#3}{XCatepillan,MPtak,WSzymanski1994}{
      \BIB{#2}{X. Catepill\'{a}n, M. Ptak, W. Szyma\'{n}ski}
         {Multiple canonical decompositions of families of operators and a model of quasinormal families}
         {\jRN{PAMS}}{121}{1994}{1165--1172}{#1}}
   \ITEE{#3}{RCauty1994}{
      \BIB{#2}{R. Cauty}
         {Un espace m\'{e}trique lin\'{e}aire qui n'est pas un r\'{e}tracte absolu}
         {\jRN{FM}}{146}{1994}{85--99, (French)}{#1}}
   \ITEE{#3}{TAChapman1971}{
      \BIB{#2}{T.A. Chapman}
         {Deficiency in infinite-dimensional manifolds}
         {\jRN{GTopA}}{1}{1971}{263--272}{#1}}
   \ITEE{#3}{TAChapman1976}{
      \BIb{#2}{T.A. Chapman}
         {Lectures on Hilbert cube manifolds}
         {C.B.M.S. Regional Conference Series in Math. No 28, Amer. Math. Soc.}{1976}{#1}}
   \ITEE{#3}{RBChuaqui1977}{
      \BIB{#2}{R.B. Chuaqui}
         {Measures invariant under a group of transformations}
         {\jRN{PacJM}}{68}{1977}{313--329}{#1}}
   \ITEE{#3}{JBConway1985}{
      \BIb{#2}{J.B. Conway}
         {A Course in Functional Analysis}
         {Springer-Verlag, New York}{1985}{#1}}
   \ITEE{#3}{JBConway2000}{
      \BIb{#2}{J.B. Conway}
         {A Course in Operator Theory}
         {(Graduate Studies in Mathematics, vol. 21) Amer. Math. Soc., Providence}{2000}{#1}}
   \ITEE{#3}{GCorach,AMaestripieri,MMbekhta2009}{
      \BIB{#2}{G. Corach, A. Maestripieri, M. Mbekhta}
         {Metric and homogeneous structure of closed range operators}
         {\jRN{JOT}}{61}{2009}{171--190}{#1}}
   \ITEE{#3}{MJCowen,RGDouglas1978}{
      \BIB{#2}{M.J. Cowen and R.G. Douglas}
         {Complex geometry and operator theory}
         {\jRN{ActaM}}{141}{1978}{187--261}{#1}}
   \ITEE{#3}{DWCurtis1985}{
      \BIB{#2}{D.W. Curtis}
         {Boundary sets in the Hilbert cube}
         {\jRN{TopA}}{20}{1985}{201--221}{#1}}
   \ITEE{#3}{MMDay1958}{
      \BIb{#2}{M.M. Day}
         {Normed Linear Spaces}
         {Springer Verlag, Berlin}{1958}{#1}}
   \ITEE{#3}{CDellacherie1967}{
      \BIB{#2}{C. Dellacherie}
         {Un compl\'{e}ment au th\'{e}or\`{e}me de Weierstrass-Stone}
         {\jRN{SeminProbStras}}{1}{1967}{52--53}{#1}}
   \ITEE{#3}{JJDijkstra1987}{
      \BIB{#2}{J.J. Dijkstra}
         {Strong negligibility of $\sigma$-compacta does not characterize Hilbert space}
         {\jRN{PacJM}}{127}{1987}{19--30}{#1}}
   \ITEE{#3}{JJDijkstra1990}{
      \BIB{#2}{J.J. Dijkstra}
         {Characterizing Hilbert space topology in terms of strong negligibility}
         {\jRN{ComposM}}{75}{1990}{299--306}{#1}}
   \ITEE{#3}{TDobrowolski,WMarciszewski2002}{
      \BIB{#2}{T. Dobrowolski and W. Marciszewski}
         {Failure of the Factor Theorem for Borel pre-Hilbert spaces}
         {\jRN{FM}}{175}{2002}{53--68}{#1}}
   \ITEE{#3}{TDobrowolski,JMogilski1990}{
      \BiB{#2}{T. Dobrowolski and J. Mogilski}
         {Problems on Topological Classification of Incomplete Metric Spaces}{Chapter 25 in:}
         {Open Problems in Topology}{J. van Mill and G.M. Reed (eds.), North-Holland Amsterdam}{1990}{411--429}{#1}}
   \ITEE{#3}{TDobrowolski,HTorunczyk1981}{
      \BIB{#2}{T. Dobrowolski and H. Toru\'{n}czyk}
         {Separable complete ANR's admitting a group structure are Hilbert manifolds}
         {\jRN{TopA}}{12}{1981}{229--235}{#1}}
   \ITEE{#3}{RGDouglas1966}{
      \BIB{#2}{R.G. Douglas}
         {On majorization, factorization and range inclusion of operators in Hilbert space}
         {\jRN{PAMS}}{17}{1966}{413--416}{#1}}
   \ITEE{#3}{CHDowker1947}{
      \BIB{#2}{C.H. Dowker}
         {Mapping theorems for non-compact spaces}
         {\jRN{AmJM}}{69}{1947}{200--242}{#1}}
   \ITEE{#3}{CHDowker1952}{
      \BIB{#2}{C.H. Dowker}
         {Topology of metric complexes}
         {\jRN{AmJM}}{74}{1952}{555--577}{#1}}
   \ITEE{#3}{JDugundji1951}{
      \BIB{#2}{J. Dugundji}
         {An extension of Tietze's theorem}
         {\jRN{PacJM}}{1}{1951}{353--367}{#1}}
   \ITEE{#3}{JDugundji1958}{
      \BIB{#2}{J. Dugundji}
         {Absolute neighborhood retracts and local connectedness for arbitrary metric spaces}
         {\jRN{ComposM}}{13}{1958}{229--246}{#1}}
   \ITEE{#3}{JDugundji1965}{
      \BIB{#2}{J. Dugundji}
         {Locally equiconnected spaces and absolute neighborhood retracts}
         {\jRN{FM}}{57}{1965}{187--193}{#1}}
   \ITEE{#3}{NDunford,JTSchwartz1958}{
      \BIb{#2}{N. Dunford and J.T. Schwartz}
         {Linear Operators, part I}
         {Interscience Publishers, New York}{1958}{#1}}
   \ITEE{#3}{NDunford,JTSchwartz1963}{
      \BIb{#2}{N. Dunford and J.T. Schwartz}
         {Linear Operators, part II}
         {Interscience Publishers, New York}{1963}{#1}}
   \ITEE{#3}{NDunford,JTSchwartz1971}{
      \BIb{#2}{N. Dunford and J.T. Schwartz}
         {Linear Operators, part III}
         {Wiley-Interscience, New York}{1971}{#1}}
   \ITEE{#3}{MLEaton,MDPerlman1977}{
      \BIB{#2}{M.L. Eaton and M.D. Perlman}
         {Reflection groups, generalized Schur functions and the geometry of majorization}
         {\jRN{AnnProb}}{5}{1977}{829--860}{#1}}
   \ITEE{#3}{EGEffros1965}{
      \BIB{#2}{E.G. Effros}
         {The Borel space of von Neumann algebras on a separable Hilbert space}
         {\jRN{PacJM}}{15}{1965}{1153--1164}{#1}}
   \ITEE{#3}{EGEffros1966}{
      \BIB{#2}{E.G. Effros}
         {Global structure in von Neumann algebras}
         {\jRN{TAMS}}{121}{1966}{434--454}{#1}}
   \ITEE{#3}{REspinola,MAKhamsi2001}{
      \BiB{#2}{R. Espinola and M.A. Khamsi}
         {Introduction to hyperconvex spaces}{Chapter XIII in:}{Handbook of Metric Fixed Point Theory}
         {W.A. Kirk and B. Sims (editors), Kluwer Academic Publishers}{2001}{391--435}{#1}}
   \ITEE{#3}{PAFillmore,JPWilliams1971}{
      \BIB{#2}{P.A. Fillmore and J.P. Williams}
         {On operator ranges}
         {\jRN{AdvM}}{7}{1971}{254--281}{#1}}
   \ITEE{#3}{JEells,NHKuiper1969}{
      \BIB{#2}{J. Eells and N.H. Kuiper}
         {Homotopy negligible subsets in infinite-dimensional manifolds}
         {\jRN{ComposM}}{21}{1969}{151--161}{#1}}
   \ITEE{#3}{REngelking1977}{
      \BIb{#2}{R. Engelking}
         {General Topology}
         {\jRN{PWN}}{1977}{#1}}
   \ITEE{#3}{REngelking1978}{
      \BIb{#2}{R. Engelking}
         {Dimension Theory}
         {\jRN{PWN}}{1978}{#1}}
   \ITEE{#3}{REngelking1989}{
      \BIb{#2}{R. Engelking}
         {General Topology. Revised and completed edition \textup{(Sigma series in pure mathematics, vol. 6)}}
         {Heldermann Verlag, Berlin}{1989}{#1}}
   \ITEE{#3}{PErdos,RDMauldin1976}{
      \BIB{#2}{P. Erd\"{o}s and R.D. Mauldin}
         {The nonexistence of certain invariant measures}
         {\jRN{PAMS}}{59}{1976}{321--322}{#1}}
   \ITEE{#3}{JErnest1976}{
      \BIB{#2}{J. Ernest}
         {Charting the operator terrain}
         {\jRN{MAMS}}{171}{1976}{207 pp}{#1}}
   \ITEE{#3}{RHFox1943}{
      \BIB{#2}{R.H. Fox}
         {On fiber spaces, II}
         {\jRN{BAMS}}{49}{1943}{733--735}{#1}}
   \ITEE{#3}{NAFriedman1970}{
      \BIb{#2}{N.A. Friedman}
         {Introduction to ergodic theory}
         {Van Nostrand Reinhold Company}{1970}{#1}}
   \ITEE{#3}{MFujii,MKajiwara,YKato,FKubo1976}{
      \BIB{#2}{M. Fujii, M. Kajiwara, Y. Kato, F. Kubo}
         {Decompositions of operators in Hilbert spaces}
         {\jRN{MathJap}}{21}{1976}{117--120}{#1}}
   \ITEE{#3}{SGao,ASKechris2003}{
      \BIB{#2}{S. Gao and A.S. Kechris}
         {On the classification of Polish metric spaces up to isometry}
         {\jRN{MAMS}}{161}{2003}{viii+78}{#1}}
   \ITEE{#3}{MIGarrido,FMontalvo1991}{
      \BIB{#2}{M.I. Garrido and F. Montalvo}
         {On some generalizations of the Kakutani-Stone and Stone-Weierstrass theorems}
         {\jRN{ExtrM}}{6}{1991}{156--159}{#1}}
   \ITEE{#3}{LGe,JShen2002}{
      \BIB{#2}{L. Ge and J. Shen}
         {Generator problem for certain property T factors}
         {\jRN{PNAS}}{99}{2002}{565--567}{#1}}
   \ITEE{#3}{IMGelfand,MANaimark1943}{
      \BIB{#2}{I.M. Gelfand and M.A. Naimark}
         {On the embedding of normed rings into the ring of operators in Hilbert space}
         {\jRN{MSb}}{12}{1943}{197--213}{#1}}
   \ITEE{#3}{FGesztesy,MMalamud,MMitrea,SNaboko2009}{
      \BIB{#2}{F. Gesztesy, M. Malamud, M. Mitrea, S. Naboko}
         {Generalized polar decompositions for closed operators in Hilbert spaces and some applications}
         {\jRN{IEOT}}{64}{2009}{83--113}{#1}}
   \ITEE{#3}{LGillman,MJerison1960}{
      \BIb{#2}{L. Gillman and M. Jerison}
         {Rings of continuous functions}
         {New York}{1960}{#1}}
   \ITEE{#3}{JGlimm1960}{
      \BIB{#2}{J. Glimm}
         {A Stone-Weierstrass theorem for $\CCc^*$-algebras}
         {\jRN{AnnM}}{72}{1960}{216--244}{#1}}
   \ITEE{#3}{GGodefroy,NJKalton2003}{
      \BIB{#2}{G. Godefroy and N.J. Kalton}
         {Lipschitz-free Banach spaces}
         {\jRN{SM}}{159}{2003}{121--141}{#1}}
   \ITEE{#3}{ICGohberg,MGKrein1967}{
      \BIB{#2}{I.C. Gohberg and M.G. Krein}
         {On a description of contraction operators similar to unitary ones}
         {\jRN{FunkAnalPril}}{1}{1967}{38--60}{#1}}
   \ITEE{#3}{ELGriffinJr1953}{
      \BIB{#2}{E.L. Griffin Jr.}
         {Some contributions to the theory of rings of operators}
         {\jRN{TAMS}}{75}{1953}{471--504}{#1}}
   \ITEE{#3}{ELGriffinJr1955}{
      \BIB{#2}{E.L. Griffin Jr.}
         {Some contributions to the theory of rings of operators II}
         {\jRN{TAMS}}{79}{1955}{389--400}{#1}}
   \ITEE{#3}{MGromov1981}{
      \BIB{#2}{M. Gromov}
         {Groups of polynomial growth and expanding maps}
         {\jRN{InHauEtSPM}}{53}{1981}{53--73}{#1}}
   \ITEE{#3}{MGromov1999}{
      \BIb{#2}{M. Gromov}
         {Metric Structures for Riemannian and Non-Riemannian Spaces}
         {Progress in Math. \textbf{152}, Birkh\"{a}user}{1999}{#1}}
   \ITEE{#3}{JDeGroot1956}{
      \BIB{#2}{J. de Groot}
         {Non-archimedean metrics in topology}
         {\jRN{PAMS}}{7}{1956}{948--953}{#1}}
   \ITEE{#3}{LCGrove,CTBenson1985}{
      \BIb{#2}{L.C. Grove and C.T. Benson}
         {Finite Reflection Group}
         {2nd ed., Springer-Verlag}{1985}{#1}}
   \ITEE{#3}{VIGurarii1966}{
      \BIB{#2}{V.I. Gurari\v{\i}}{Spaces of universal placement, isotropic spaces and a problem of Mazur 
         on rotations of Banach spaces \textup{(Russian)}}
         {\jRN{SibirMZ}}{7}{1966}{1002--1013}{#1}}
   \ITEE{#3}{DWHadwin1976}{
      \BIB{#2}{D.W. Hadwin}
         {An operator-valued spectrum}
         {\jRN{NAMS}}{23}{1976}{A-163}{#1}}
   \ITEE{#3}{DWHadwin1977}{
      \BIB{#2}{D.W. Hadwin}
         {An operator-valued spectrum}
         {\jRN{IndianaUMJ}}{26}{1977}{329--340}{#1}}
   \ITEE{#3}{HHahn1932}{
      \BIb{#2}{H. Hahn}
         {Reelle Funktionen I}
         {Leipzig}{1932}{#1}}
   \ITEE{#3}{PRHalmos1950}{
      \BIb{#2}{P.R. Halmos}
         {Measure theory}
         {Van Nostrand, New York}{1950}{#1}}
   \ITEE{#3}{PRHalmos1951}{
      \BIb{#2}{P.R. Halmos}
         {Introduction to Hilbert Space and the Theory of Spectral Multiplicity}
         {Chelsea Publishing Company, New York}{1951}{#1}}
   \ITEE{#3}{PRHalmos1956}{
      \BIb{#2}{P.R. Halmos}
         {Lectures on Ergodic Theory}
         {Publ. Math. Soc. Japan, Tokyo}{1956}{#1}}
   \ITEE{#3}{PRHalmos1982}{
      \BIb{#2}{P.R. Halmos}
         {A Hilbert Space Problem Book}
         {Springer-Verlag New York Inc.}{1982}{#1}}
  \ITEE{#3}{PRHalmos,JEMcLaughlin1963}{
      \BIB{#2}{P.R. Halmos and J.E. McLaughlin}
         {Partial isometries}
         {\jRN{PacJM}}{13}{1963}{585--596}{#1}}
   \ITEE{#3}{RWHansell1972}{
      \BIB{#2}{R.W. Hansell}
         {On the nonseparable theory of Borel and Souslin sets}
         {\jRN{BAMS}}{78}{1972}{236--241}{#1}}
   \ITEE{#3}{FHausdorff1930}{
      \BIB{#2}{F. Hausdorff}
         {Erweiterung einer Hom\"{o}omorphie}
         {\jRN{FM}}{16}{1930}{353--360}{#1}}
   \ITEE{#3}{FHausdorff1934}{
      \BIB{#2}{F. Hausdorff}
         {\"{U}ber innere Abbildungen}
         {\jRN{FM}}{23}{1934}{279--291}{#1}}
   \ITEE{#3}{FHausdorff1938}{
      \BIB{#2}{F. Hausdorff}
         {Erweiterung einer stetigen Abbildung}
         {\jRN{FM}}{30}{1938}{40--47}{#1}}
   \ITEE{#3}{DWHenderson1971}{
      \BIB{#2}{D.W. Henderson}
         {Corrections and extensions of two papers about infinite-dimensional manifolds}
         {\jRN{GTopA}}{1}{1971}{321--327}{#1}}
   \ITEE{#3}{DWHenderson1975}{
      \BIB{#2}{D.W. Henderson}
         {$Z$-sets in ANR's}
         {\jRN{TAMS}}{213}{1975}{205--216}{#1}}
   \ITEE{#3}{DWHenderson,RMSchori1970}{
      \BIB{#2}{D.W. Henderson and R.M. Schori}
         {Topological classification of infinite-dimensional manifolds by homotopy type}
         {\jRN{BAMS}}{76}{1970}{121--124}{#1}}
   \ITEE{#3}{DWHenderson,JEWest1970}{
      \BIB{#2}{D.W. Henderson and J.E. West}
         {Triangulated infinite-dimensional manifolds}
         {\jRN{BAMS}}{76}{1970}{655--660}{#1}}
   \ITEE{#3}{BHoffmann1979}{
      \BIB{#2}{B. Hoffmann}
         {A compact contractible topological group is trivial}
         {\jRN{ArchM}}{32}{1979}{585--587}{#1}}
   \ITEE{#3}{DHofmann2002}{
      \BIB{#2}{D. Hofmann}
         {On a generalization of the Stone-Weierstrass theorem}
         {\jRN{ACS}}{10}{2002}{569--592}{#1}}
   \ITEE{#3}{GHognas,AMukherjea1995}{
      \BIb{#2}{G. H\"ogn\"as and A. Mukherjea}
         {Probability Measures on Semigroups. Convolution Products, Random Walks, and Random Matrices}
         {Plenum Press, New York}{1995}{#1}}
   \ITEE{#3}{MRHolmes1992}{
      \BIB{#2}{M.R. Holmes}
         {The universal separable metric space of Urysohn and isometric embeddings thereof in Banach spaces}
         {\jRN{FM}}{140}{1992}{199--223}{#1}}
   \ITEE{#3}{MRHolmes2008}{
      \BIB{#2}{M.R. Holmes}
         {The Urysohn space embeds in Banach spaces in just one way}
         {\jRN{TopA}}{155}{2008}{1479--1482}{#1}}
   \ITEE{#3}{RRHolmes,TYTam1999}{
      \BIB{#2}{R.R. Holmes and T.Y. Tam}
         {Distance to the convex hull of an orbit under the action of a compact group}
         {\jRN{JAusMSA}}{66}{1999}{331--357}{#1}}
   \ITEE{#3}{RHorn,RMathias1990}{
      \BIB{#2}{R. Horn and R. Mathias}
         {Cauchy-Schwartz inequalities associated with positive semidefinite matrices}
         {\jRN{LAA}}{142}{1990}{63--82}{#1}}
   \ITEE{#3}{GEHuhunaisvili1955}{
      \BIB{#2}{G.E. Huhunai\v{s}vili}
         {On a property of Urysohn's universal metric space}
         {\jRN{DANSSSR}}{101}{1955}{607--610 (Russian)}{#1}}
   \ITEE{#3}{JEHumphreys1990}{
      \BIb{#2}{J.E. Humphreys}
         {Reflection Groups and Coxeter Groups}
         {Cambridge University Press}{1990}{#1}}
   \ITEE{#3}{JRIsbell1964}{
      \BIB{#2}{J.R. Isbell}
         {Six theorems about injective metric spaces}
         {\jRN{CMHelv}}{39}{1964}{65--76}{#1}}
   \ITEE{#3}{SIzumino,YKato1985}{
      \BIB{#2}{S. Izumino and Y. Kato}
         {The closure of invertible operators on Hilbert space}
         {\jRN{ActaSM}}{49}{1985}{321--327}{#1}}
   \ITEE{#3}{CJiang2004}{
      \BIB{#2}{C. Jiang}
         {Similarity classification of Cowen-Douglas operators}
         {\jRN{CanadJM}}{56}{2004}{742--775}{#1}}
   \ITEE{#3}{WBJohnson,JLindenstrauss2001}{
      \BiB{#2}{W.B. Johnson and J. Lindenstrauss}{Basic Concepts in the Geometry of Banach Spaces}
         {Chapter 1 in:}{Handbook of the Geometry of Banach Spaces, Vol. 1}
         {W.B. Johnson and J. Lindenstrauss (editors), Elsevier Science B.V., Amsterdam}{2001}{1--84}{#1}}
   \ITEE{#3}{IBJung,JStochel2008}{
      \BIB{#2}{I.B. Jung and J. Stochel}
         {Subnormal operators whose adjoints have rich point spectrum}
         {\jRN{JFA}}{255}{2008}{1797--1816}{#1}}
   \ITEE{#3}{RVKadison,JRRingrose1983}{
      \BIb{#2}{R.V. Kadison and J.R. Ringrose}
         {Fundamentals of the Theory of Operator Algebras. Volume I: Elementary Theory}
         {Academic Press, Inc., New York-London}{1983}{#1}}
   \ITEE{#3}{RVKadison,JRRingrose1986}{
      \BIb{#2}{R.V. Kadison and J.R. Ringrose}
         {Fundamentals of the Theory of Operator Algebras. Volume II: Advanced Theory}
         {Academic Press, Inc., Orlando-London}{1986}{#1}}
   \ITEE{#3}{SKakutani1936}{
      \BIB{#2}{S. Kakutani}
         {\"{U}ber die Metrisation der topologischen Gruppen}
         {\jRN{ProcImpAcadTokyo}}{12}{1936}{82--84}{#1}}
   \ITEE{#3}{SKakutani1938}{
      \BIB{#2}{S. Kakutani}
         {Two fixed-point theorems concerning bicompact convex sets}
         {\jRN{ProcImpAcadTokyo}}{14}{1938}{242--245}{#1}}
   \ITEE{#3}{SKakutani1941}{
      \BIB{#2}{S. Kakutani}
         {Concrete representation of abstract L-spaces}
         {\jRN{AnnM}}{42}{1941}{523--537}{#1}}
   \ITEE{#3}{SKakutani1941a}{
      \BIB{#2}{S. Kakutani}
         {Concrete representation of abstract M-spaces}
         {\jRN{AnnM}}{42}{1941}{994--1024}{#1}}
   \ITEE{#3}{NKalton2007}{
      \BIB{#2}{N. Kalton}
         {Extending Lipschitz maps into $\CCc(K)$-spaces}
         {\jRN{IsraelJM}}{162}{2007}{275--315}{#1}}
   \ITEE{#3}{RKane2001}{
      \BIb{#2}{R. Kane}
         {Reflection Groups and Invariant Theory}
         {Canadian Mathematical Society, Springer}{2001}{#1}}
   \ITEE{#3}{VKannan,SRRaju1980}{
      \BIB{#2}{V. Kannan and S.R. Raju}
         {The nonexistence of invariant universal measures on semigroups}
         {\jRN{PAMS}}{78}{1980}{482--484}{#1}}
   \ITEE{#3}{IKaplansky1951}{
      \BIB{#2}{I. Kaplansky}
         {A theorem on rings of operators}
         {\jRN{PacJM}}{1}{1951}{227--232}{#1}}
   \ITEE{#3}{MKatetov1988}{
      \BiB{#2}{M. Kat\v{e}tov}{On universal metric spaces}{in: Frolik (ed.),}
         {General Topology and its Relations to Modern Analysis and Algebra VI. Proceedings of the Sixth Prague 
         Topological Symposium 1986}{Heldermann Verlag Berlin}{1988}{323--330}{#1}}
   \ITEE{#3}{YKatznelson1960}{
      \BIB{#2}{Y. Katznelson}
         {Sur les alg\'{e}bres dont les \'{e}l\'{e}ments non n\'{e}gatifs admettent des racines carr\'{e}es}
         {\jRN{AnnSciEcNormSupT}}{77}{1960}{167--174}{#1}}
   \ITEE{#3}{OHKeller1931}{
      \BIB{#2}{O.H. Keller}
         {Die Homoiomorphie der kompakten konvexen Mengen in Hilbertschen Raum}
         {\jRN{MAnn}}{105}{1931}{748--758}{#1}}
   \ITEE{#3}{MAKhamsi,WAKirk,CMartinez2000}{
      \BIB{#2}{M.A. Khamsi, W.A. Kirk, C. Martinez}
         {Fixed point and selection theorems in hyperconvex spaces}
         {\jRN{PAMS}}{128}{2000}{3275--3283}{#1}}
   \ITEE{#3}{ABKhararazishvili1998}{
      \BIb{#2}{A.B. Khararazishvili}
         {Transformation groups and invariant measures. Set-theoretic aspects}
         {World Scientific Publishing Co., Inc., River Edge, NJ}{1998}{#1}}
   \ITEE{#3}{YKijima1987}{
      \BIB{#2}{Y. Kijima}
         {Fixed points of nonexpansive self-maps of a compact metric space}
         {\jRN{JMAnApp}}{123}{1987}{114--116}{#1}}
  \ITEE{#3}{JSKim,ChRKim,SGLee1980}{
      \BIB{#2}{J.S. Kim, Ch.R. Kim, S.G. Lee}
         {Reducing operator valued spectra of a Hilbert space operator}
         {\jRN{JKoreanMS}}{17}{1980}{123--129}{#1}}
   \ITEE{#3}{JKindler1995}{
      \BIB{#2}{J. Kindler}
         {Minimax theorems with applications to convex metric spaces}
         {\jRN{CollM}}{68}{1995}{179--186}{#1}}
   \ITEE{#3}{WAKirk1998}{
      \BIB{#2}{W.A. Kirk}
         {Hyperconvexity of $\RRR$-trees}
         {\jRN{FM}}{156}{1998}{67--72}{#1}}
   \ITEE{#3}{VLKleeJr1952}{
      \BIB{#2}{V.L. Klee Jr.}
         {Invariant metrics in groups (solution of a problem of Banach)}
         {\jRN{PAMS}}{3}{1952}{484--487}{#1}}
   \ITEE{#3}{HJKowalsky1957}{
      \BIB{#2}{H.J. Kowalsky}
         {Einbettung metrischer R\"{a}ume}
         {\jRN{ArchM}}{8}{1957}{336--339}{#1}}
   \ITEE{#3}{WKubis,MRubin2010}{
      \BIB{#2}{W. Kubi\'{s} and M. Rubin}
         {Extension and reconstruction theorems for the Urysohn universal metric space}
         {\jRN{CzMJ}}{60}{2010}{1--29}{#1}}
   \ITEE{#3}{KKuratowski1966}{
      \BIb{#2}{K. Kuratowski}
         {Topology. \textup{Vol. I}}
         {\jRN{PWN}}{1966}{#1}}
   \ITEE{#3}{KKuratowski,BKnaster1927}{
      \BIB{#2}{K. Kuratowski and B. Knaster}
         {A connected and connected im kleinen point set which contains no perfect subset}
         {\jRN{BAMS}}{33}{1927}{106--109}{#1}}
   \ITEE{#3}{KKuratowski,AMostowski1976}{
      \BIb{#2}{K. Kuratowski and A. Mostowski}
         {Set Theory with an Introduction to Descriptive Set Theory}
         {\jRN{PWN}}{1976}{#1}}
   \ITEE{#3}{GLewicki1992}{
      \BIB{#2}{G. Lewicki}
         {Bernstein's ``lethargy'' theorem in metrizable topological linear spaces}
         {\jRN{MonatM}}{113}{1992}{213--226}{#1}}
   \ITEE{#3}{ASLewis1996}{
      \BIB{#2}{A.S. Lewis}
         {Group invariance and convex matrix analysis}
         {\jRN{SIAMJMAA}}{17}{1996}{927--949}{#1}}
   \ITEE{#3}{C-KLi,N-KTsing1991}{
      \BIB{#2}{C.-K. Li and N.-K. Tsing}
         {$G$-invariant norms and $G(c)$-radii}
         {\jRN{LAA}}{150}{1991}{179--194}{#1}}
   \ITEE{#3}{AJLazar,JLindenstrauss1971}{
      \BIB{#2}{A.J. Lazar and J. Lindenstrauss}
         {Banach spaces whose duals are $L_1$ spaces and their representing matrices}
         {\jRN{ActaM}}{126}{1971}{165--193}{#1}}
   \ITEE{#3}{EHLieb,MLoss1997}{
      \BIb{#2}{E.H. Lieb and M. Loss}
         {Analysis \textup{(Graduate Studies in Mathematics, vol. 14)}}
         {Amer. Math. Soc., Providence, RI}{1997}{#1}}
   \ITEE{#3}{ALindenbaum1926}{
      \BIB{#2}{A. Lindenbaum}
         {Contributions \`{a} l'\'{e}tude de l'espace m\'{e}trique I}
         {\jRN{FM}}{8}{1926}{209--222}{#1}}
   \ITEE{#3}{DLindenstrauss,LTzafriri1971}{
      \BIB{#2}{D. Lindenstrauss and L. Tzafriri}
         {On the complemented subspaces problem}
         {\jRN{IsraelJM}}{9}{1971}{263--269}{#1}}
   \ITEE{#3}{RILoebl1986}{
      \BIB{#2}{R.I. Loebl}
         {A note on containment of operators}
         {\jRN{BAustrMS}}{33}{1986}{279--291}{#1}}
   \ITEE{#3}{LHLoomis1945}{
      \BIB{#2}{L.H. Loomis}
         {Abstract congruence and the uniqueness of Haar measure}
         {\jRN{AnnM}}{46}{1945}{348--355}{#1}}
   \ITEE{#3}{LHLoomis1949}{
      \BIB{#2}{L.H. Loomis}
         {Haar measure in uniform structures}
         {\jRN{DukeMJ}}{16}{1949}{193--208}{#1}}
   \ITEE{#3}{ERLorch1939}{
      \BIB{#2}{E.R. Lorch}
         {Bicontinuous linear transformation in certain vector spaces}
         {\jRN{BAMS}}{45}{1939}{564--569}{#1}}
   \ITEE{#3}{ATLundell,SWeingram1969}{
      \BIb{#2}{A.T. Lundell and S. Weingram}
         {The topology of CW-complexes}
         {Litton Educ. Publ.}{1969}{#1}}
   \ITEE{#3}{WLusky1976}{
      \BIB{#2}{W. Lusky}
         {The Gurarij spaces are unique}
         {\jRN{ArchM}}{27}{1976}{627--635}{#1}}
   \ITEE{#3}{WLusky1977}{
      \BIB{#2}{W. Lusky}
         {On separable Lindenstrauss spaces}
         {\jRN{JFA}}{26}{1977}{103--120}{#1}}
   \ITEE{#3}{DMaharam1942}{
      \BIB{#2}{D. Maharam}
         {On homogeneous measure algebras}
         {\jRN{PNatlUSA}}{28}{1942}{108--111}{#1}}
   \ITEE{#3}{MMalicki,SSolecki2009}{
      \BIB{#2}{M. Malicki and S. Solecki}
         {Isometry groups of separable metric spaces}
         {\jRN{MProcCambPhS}}{146}{2009}{67--81}{#1}}
   \ITEE{#3}{PMankiewicz1972}{
      \BIB{#2}{P. Mankiewicz}
         {On extension of isometries in normed linear spaces}
         {\jRN{BAPolSSSM}}{20}{1972}{367--371}{#1}}
   \ITEE{#3}{JMartinezMaurica,MTPellon1987}{
      \BIB{#2}{J. Martinez-Maurica and M.T. Pell\'{o}n}
         {Non-archimedean Chebyshev centers}
         {\jRN{IndagMP}}{90}{1987}{417--421}{#1}}
   \ITEE{#3}{KMaurin1980}{
      \BIb{#2}{K. Maurin}
         {Analysis, Part II}
         {D. Reidel, Dordrecht-Boston-London}{1980}{#1}}
   \ITEE{#3}{SMazur,SUlam1932}{
      \BIB{#2}{S. Mazur and S. Ulam}
         {Sur les transformationes isom\'{e}triques d'espaces vectoriels norm\'{e}s}
         {\jRN{CRASParis}}{194}{1932}{946--948}{#1}}
   \ITEE{#3}{SMazurkiewicz1920}{
      \BIB{#2}{S. Mazurkiewicz}
         {Sur les lignes de Jordan}
         {\jRN{FM}}{1}{1920}{166--209}{#1}}
   \ITEE{#3}{SMazurkiewicz,WSierpinski1920}{
      \BIB{#2}{S. Mazurkiewicz and W. Sierpi\'{n}ski}
         {Contributions a la topologie des ensembles denombrables}
         {\jRN{FM}}{1}{1920}{17--27}{#1}}
   \ITEE{#3}{MMbekhta1992}{
      \BIB{#2}{M. Mbekhta}
         {Sur la structure des composantes connexes semi-Fredholm de $B(H)$}
         {\jRN{PAMS}}{116}{1992}{521--524}{#1}}
   \ITEE{#3}{JEMcCarthy1996}{
      \BIB{#2}{J.E. McCarthy}
         {Boundary values and Cowen-Douglas curvature}
         {\jRN{JFA}}{137}{1996}{1--18}{#1}}
   \ITEE{#3}{JMelleray2007}{
      \BIB{#2}{J. Melleray}
         {Computing the complexity of the relation of isometry between separable Banach spaces}
         {\jRN{MLQ}}{53}{2007}{128--131}{#1}}
   \ITEE{#3}{JMelleray2007a}{
      \BIB{#2}{J. Melleray}
         {On the geometry of Urysohn's universal metric space}
         {\jRN{TopA}}{154}{2007}{384--403}{#1}}
   \ITEE{#3}{JMelleray2008}{
      \BIB{#2}{J. Melleray}
         {Some geometric and dynamical properties of the Urysohn space}
         {\jRN{TopA}}{155}{2008}{1531--1560}{#1}}
   \ITEE{#3}{JMelleray,FVPetrov,AMVershik2008}{
      \BIB{#2}{J. Melleray, F.V. Petrov, A.M. Vershik}
         {Linearly rigid metric spaces and the embedding problem}
         {\jRN{FM}}{199}{2008}{177--194}{#1}}
   \ITEE{#3}{EMichael1953}{
      \BIB{#2}{E. Michael}
         {Some extension theorems for continuous functions}
         {\jRN{PacJM}}{3}{1953}{789--806}{#1}}
   \ITEE{#3}{EMichael1954}{
      \BIB{#2}{E. Michael}
         {Local properties of topological spaces}
         {\jRN{DukeMJ}}{21}{1954}{163--171}{#1}}
   \ITEE{#3}{EMichael1956}{
      \BIB{#2}{E. Michael}
         {Selected selection theorems}
         {\jRN{AmMMon}}{58}{1956}{233--238}{#1}}
   \ITEE{#3}{EMichael1956a}{
      \BIB{#2}{E. Michael}
         {Continuous selections. I}
         {\jRN{AnnM}}{63}{1956}{361--382}{#1}}
   \ITEE{#3}{EMichael1956b}{
      \BIB{#2}{E. Michael}
         {Continuous selections. II}
         {\jRN{AnnM}}{64}{1956}{562--580}{#1}}
   \ITEE{#3}{EMichael1959}{
      \BIB{#2}{E. Michael}
         {A theorem on semi-continuous set-valued functions}
         {\jRN{DukeMJ}}{26}{1959}{647--652}{#1}}
   \ITEE{#3}{JVanMill1986}{
      \BIB{#2}{J. van Mill}
         {Another counterexample in ANR theory}
         {\jRN{PAMS}}{97}{1986}{136--138}{#1}}
   \ITEE{#3}{JVanMill2001}{
      \BIb{#2}{J. van Mill}
         {The Infinite-Dimensional Topology of Function Spaces 
         \textup{(North-Holland Mathematical Library, vol. 64)}}
         {Elsevier, Amsterdam}{2001}{#1}}
   \ITEE{#3}{WMlak1991}{
      \BIb{#2}{W. Mlak}
         {Hilbert Spaces and Operator Theory}
         {PWN --- Polish Scientific Publishers and Kluwer Academic Publishers, Warszawa-Dordrecht}{1991}{#1}}
   \ITEE{#3}{JMogilski1979}{
      \BIB{#2}{J. Mogilski}
         {$CE$-decomposition of $l_2$-manifolds}
         {\jRN{BAPolSSSM}}{27}{1979}{309--314}{#1}}
   \ITEE{#3}{RLMoore1916}{
      \BIB{#2}{R.L. Moore}
         {On the foundations of plane analysis situs}
         {\jRN{TAMS}}{17}{1916}{131--164}{#1}}
   \ITEE{#3}{KMorita1955}{
      \BIB{#2}{K. Morita}
         {A condition for the metrizability of topological spaces and for $n$-dimensionality}
         {\jRN{SciRepTokyoA}}{5}{1955}{33--36}{#1}}
   \ITEE{#3}{AMukherjea,NATserpes1976}{
      \BIb{#2}{A. Mukherjea and N.A. Tserpes}
         {Measures on topological semigroups}
         {Springer Lecture Notes in Math. Vol. 547, Berlin}{1976}{#1}}
   \ITEE{#3}{JMycielski1974}{
      \BIB{#2}{J. Mycielski}
         {Remarks on invariant measures in metric spaces}
         {\jRN{CollM}}{32}{1974}{105--112}{#1}}
   \ITEE{#3}{SNNaboko1984}{
      \BIB{#2}{S.N. Naboko}
         {Conditions for similarity to unitary and selfadjoint operators}
         {\jRN{FunkAnalPril}}{18}{1984}{16--27}{#1}}
   \ITEE{#3}{LNachbin1965}{
      \BIb{#2}{L. Nachbin}
         {The Haar Integral}
         {D. Van Nostrand Company, Inc., Princeton-New Jersey-Toronto-New York-London}{1965}{#1}}
   \ITEE{#3}{TDNarang,SKGarg1991}{
      \BIB{#2}{T.D. Narang and S.K. Garg}
         {On the uniqueness of best approximation in non-archimedian spaces}
         {\jRN{PeriodMHung}}{22}{1991}{121--124}{#1}}
   \ITEE{#3}{JVonNeumann1930}{
      \BIB{#2}{J. von Neumann}
         {Zur Algebra der Funktionaloperationen und Theorie der normalen Operatoren}
         {\jRN{MAnn}}{102}{1930}{370--427}{#1}}
   \ITEE{#3}{JVonNeumann1934}{
      \BIB{#2}{J. von Neumann}
         {Zum Haarschen Mass in topologischen Gruppen}
         {\jRN{ComposM}}{1}{1934}{106--114}{#1}}
   \ITEE{#3}{JVonNeumann1937}{
      \BiB{#2}{J. von Neumann}
         {Some matrix-inequalities and metrization of matrix-space}{\jRN{TomskUnivRev}{} \textbf{1} (1937), 286--300; 
         in }{Collected Works}{Pergamon, New York}{1962}{Vol. 4, 205--219}{#1}}
   \ITEE{#3}{JVonNeumann1949}{
      \BIB{#2}{J. von Neumann}
         {On Rings of Operators. Reduction Theory}
         {\jRN{AnnM}}{50}{1949}{401--485}{#1}}
   \ITEE{#3}{ONielson1973}{
      \BIB{#2}{O. Nielson}
         {Borel sets of von Neumann algebras}
         {\jRN{AmJM}}{95}{1973}{145--164}{#1}}
   \ITEE{#3}{pn2002}{\bibITEM{#2}{#1} \mypaplist{pn1}}
   \ITEE{#3}{pn2006a}{\bibITEM{#2}{#1} \mypaplist{pn2}}
   \ITEE{#3}{pn2006b}{\bibITEM{#2}{#1} \mypaplist{pn3}}
   \ITEE{#3}{pn2007}{\bibITEM{#2}{#1} \mypaplist{pn4}}
   \ITEE{#3}{pn2008a}{\bibITEM{#2}{#1} \mypaplist{pn5}}
   \ITEE{#3}{pn2008b}{\bibITEM{#2}{#1} \mypaplist{pn6}}
   \ITEE{#3}{pn2009a}{\bibITEM{#2}{#1} \mypaplist{pn7}}
   \ITEE{#3}{pn2009b}{\bibITEM{#2}{#1} \mypaplist{pn8}}
   \ITEE{#3}{pn2009c}{\bibITEM{#2}{#1} \mypaplist{pn9}}
   \ITEE{#3}{pn2010a}{\bibITEM{#2}{#1} \mypaplist{pn12}}
   \ITEE{#3}{pn2010b}{\bibITEM{#2}{#1} \mypaplist{pn13}}
   \ITEE{#3}{pn2011a}{\bibITEM{#2}{#1} \mypaplist{pn10}}
   \ITEE{#3}{pn2011b}{\bibITEM{#2}{#1} \mypaplist{pn15}}
   \ITEE{#3}{pn2011c}{\bibITEM{#2}{#1} \mypaplist{pn16}}
   \ITEE{#3}{pn2011d}{\bibITEM{#2}{#1} \mypaplist{pn17}}
   \ITEE{#3}{pn2009x}{
      \bibITEM{#2}{#1} \mypaplist{pn11}}
   \ITEE{#3}{pn2010x}{
      \bibITEM{#2}{#1} \mypaplist{pn14}}
   \ITEE{#3}{pnXXXXb}{
      \bibITEM{#2}{#1} \mypaplist{pnX2}}
   \ITEE{#3}{pnXXXXc}{
      \bibITEM{#2}{#1} \mypaplist{pnX3}}
   \ITEE{#3}{pnXXXXd}{
      \bibITEM{#2}{#1} \mypaplist{pnX13}}
   \ITEE{#3}{MNiezgoda1998}{
      \BIB{#2}{M. Niezgoda}
         {Group majorization and Schur type inequalities}
         {\jRN{LAA}}{268}{1998}{9--30}{#1}}
   \ITEE{#3}{MNiezgoda1998a}{
      \BIB{#2}{M. Niezgoda}
         {An analytical characterization of effective and of irreducible groups inducing cone orderings}
         {\jRN{LAA}}{269}{1998}{105--114}{#1}}
   \ITEE{#3}{MNiezgoda,TYTam2001}{
      \BIB{#2}{M. Niezgoda and T.Y. Tam}
         {On norm property of $G(c)$-radii and Eaton triples}
         {\jRN{LAA}}{336}{2001}{119--130}{#1}}
   \ITEE{#3}{APazy1983}{
      \BIb{#2}{A. Pazy}{Semigroups of Linear Operators 
         and Applications to Partial Differential Equations \textup{(Applied Mathematical Sciences, vol. 44)}}
         {Springer-Verlag, New York}{1983}{#1}}
   \ITEE{#3}{APelc1982}{
      \BIB{#2}{A. Pelc}
         {Semiregular invariant measures on abelian groups}
         {\jRN{PAMS}}{86}{1982}{423--426}{#1}}
   \ITEE{#3}{RPenrose1955}{
      \BIB{#2}{R. Penrose}
         {A generalized inverse for matrices}
         {\jRN{ProcCambPhS}}{51}{1955}{406--413}{#1}}
   \ITEE{#3}{VPestov2006}{
      \BIb{#2}{V. Pestov}
         {Dynamics of infinite-dimensional groups. The Ramsey-Dvoretzky-Milman phenomenon}
         {University Lecture Series \textbf{40}, AMS, Providence, RI}{2006}{#1}}
   \ITEE{#3}{VPestov2007}{
      \BiB{#2}{V. Pestov}
         {Forty-plus annotated questions about large topological groups}
         {in:}{Open Problems in Topology II}{Elliot Pearl (editor), Elsevier B.V., Amsterdam}{2007}{439--450}{#1}}
   \ITEE{#3}{PVPetersen1993}{
      \BiB{#2}{P.V. Petersen}
         {Gromov-Hausdorff convergence of metric spaces}{in book:}{Differential Geometry: Riemannian Geometry 
         (Los Angeles, CA, 1990)}{Amer. Math. Soc., Providence, RI}{1993}{489--504}{#1}}
   \ITEE{#3}{DRamachandran,MMisiurewicz1982}{
      \BIB{#2}{D. Ramachandran and M. Misiurewicz}
         {Hopf's theorem on invariant measures for a group of transformations}
         {\jRN{SM}}{74}{1982}{183--189}{#1}}
   \ITEE{#3}{JMRosenblatt1974}{
      \BIB{#2}{J.M. Rosenblatt}
         {Equivalent invariant measures}
         {\jRN{IsraelJM}}{17}{1974}{261--270}{#1}}
   \ITEE{#3}{HLRoyden1963}{
      \BIb{#2}{H.L. Royden}
         {Real Analysis}
         {The Macmillan Co., New York}{1963}{#1}}
   \ITEE{#3}{WRudin1962}{
      \BIb{#2}{W. Rudin}
         {Fourier Analysis on Groups \textup{(Interscience Tracts in Pure and Applied Mathematics, Number 12)}}
         {Interscience Publishers, New York}{1962}{#1}}
   \ITEE{#3}{WRudin1991}{
      \BIb{#2}{W. Rudin}
         {Functional Analysis}
         {McGraw-Hill Science}{1991}{#1}}
   \ITEE{#3}{TSaito1972}{
      \BiB{#2}{T. Sait\^{o}}{Generations of von Neumann algebras}
         {Lecture Notes in Math. vol. 247}{\textup{(}Lecture on Operator Algebras\textup{)}}
         {Springer, Berlin-Heidelberg-New York}{1972}{435--531}{#1}}
   \ITEE{#3}{KSakai,MYaguchi2003}{
      \BIB{#2}{K. Sakai and M. Yaguchi}
         {Characterizing manifolds modeled on certain dense subspaces of non-separable Hilbert spaces}
         {\jRN{TsukubaJM}}{27}{2003}{143--159}{#1}}
   \ITEE{#3}{SSakai1971}{
      \BIb{#2}{S. Sakai}
         {$\CCc^*$-Algebras and $\WWw^*$-Algebras}
         {Springer-Verlag, Berlin-Heidelberg-New York}{1971}{#1}}
   \ITEE{#3}{RSchori1971}{
      \BIB{#2}{R. Schori}
         {Topological stability for infinite-dimensional manifolds}
         {\jRN{ComposM}}{23}{1971}{87--100}{#1}}
   \ITEE{#3}{JTSchwartz1967}{
      \BIb{#2}{J.T. Schwartz}
         {$\WWw^*$-algebras}
         {Gordon and Breach, Science Publishers Inc., New York-London-Paris}{1967}{#1}}
   \ITEE{#3}{ZSemadeni1971}{
      \BIb{#2}{Z. Semadeni}
         {Banach Spaces of Continuous Functions (Vol. I)}
         {\jRN{PWN}}{1971}{#1}}
   \ITEE{#3}{JPSerre1951}{
      \BIB{#2}{J.-P. Serre}
         {Homologie singuli\`{e}re des espaces fibr\'{e}s}
         {\jRN{AnnM}}{54}{1951}{425--505}{#1}}
   \ITEE{#3}{DSherman2007}{
      \BIB{#2}{D. Sherman}
         {On the dimension theory of von Neumann algebras}
         {\jRN{MScand}}{101}{2007}{123--147}{#1}}
   \ITEE{#3}{WSierpinski1928}{
      \BIB{#2}{W. Sierpi\'{n}ski}
         {Sur les projections des ensembles compl\'{e}mentaires aux ensembles \textup{(A)}}
         {\jRN{FM}}{11}{1928}{117--122}{#1}}
   \ITEE{#3}{MSlocinski1980}{
      \BIB{#2}{M. S\l{}oci\'{n}ski}
         {On the Wold-type decomposition of a pair of commuting isometries}
         {\jRN{APM}}{37}{1980}{255--262}{#1}}
   \ITEE{#3}{RCSteinlage1975}{
      \BIB{#2}{R.C. Steinlage}
         {On Haar measure in locally compact $T_2$ spaces}
         {\jRN{AmJM}}{97}{1975}{291--307}{#1}}
   \ITEE{#3}{JStochel,FHSzafraniec1989}{
      \BIB{#2}{J. Stochel and F.H. Szafraniec}
         {On normal extensions of unbounded operators. III. Spectral properties}
         {\jRN{PublRIMSKyoto}}{25}{1989}{105--139}{#1}}
   \ITEE{#3}{JStochel,FHSzafraniec1989a}{
      \BIB{#2}{J. Stochel and F.H. Szafraniec}
         {The normal part of an unbounded operator}
         {\jRN{ProcKonink}}{92}{1989}{495--503}{#1}}
   \ITEE{#3}{AHStone1962}{
      \BIB{#2}{A.H. Stone}
         {Absolute $\FFf_{\sigma}$-spaces}
         {\jRN{PAMS}}{13}{1962}{495--499}{#1}}
   \ITEE{#3}{AHStone1962a}{
      \BIB{#2}{A.H. Stone}
         {Non-separable Borel sets}
         {\jRN{DissM}}{28}{1962}{41 pages}{#1}}
   \ITEE{#3}{AHStone1972}{
      \BIB{#2}{A.H. Stone}
         {Non-separable Borel sets II}
         {\jRN{GTopA}}{2}{1972}{249--270}{#1}}
   \ITEE{#3}{MHStone1937}{
      \BIB{#2}{M.H. Stone}
         {Application of the theory of Boolean rings to general topology}
         {\jRN{TAMS}}{41}{1937}{375--481}{#1}}
   \ITEE{#3}{MHStone1948}{
      \BIB{#2}{M.H. Stone}
         {The generalized Weierstrass approximation theorem}
         {\jRN{MMag}}{21}{1948}{167--184}{#1}}
   \ITEE{#3}{BSz-Nagy1947}{
      \BIB{#2}{B. Sz.-Nagy}
         {On uniformly bounded linear transformations in Hilbert space}
         {\jRN{ActaSM}}{11}{1947}{152--157}{#1}}
   \ITEE{#3}{WTakahashi1970}{
      \BIB{#2}{W. Takahashi}
         {A convexity in metric space and nonexpansive mappings, I}
         {\jRN{KodaiMSemRep}}{22}{1970}{142--149}{#1}}
   \ITEE{#3}{MTakesaki2002}{
      \BIb{#2}{M. Takesaki}
         {Theory of Operator Algebras I \textup{(Encyclopaedia of Mathematical Sciences, Volume 124)}}
         {Springer-Verlag, Berlin-Heidelberg-New York}{2002}{#1}}
   \ITEE{#3}{MTakesaki2003}{
      \BIb{#2}{M. Takesaki}
         {Theory of Operator Algebras II \textup{(Encyclopaedia of Mathematical Sciences, Volume 125)}}
         {Springer-Verlag, Berlin-Heidelberg-New York}{2003}{#1}}
   \ITEE{#3}{MTakesaki2003a}{
      \BIb{#2}{M. Takesaki}
         {Theory of Operator Algebras III \textup{(Encyclopaedia of Mathematical Sciences, Volume 127)}}
         {Springer-Verlag, Berlin-Heidelberg-New York}{2003}{#1}}
   \ITEE{#3}{TYTam1999}{
      \BIB{#2}{T.Y. Tam}
         {An extension of a result of Lewis}
         {\jRN{ELA}}{5}{1999}{1--10}{#1}}
   \ITEE{#3}{TYTam2000}{
      \BIB{#2}{T.Y. Tam}
         {Group majorization, Eaton triples and numerical range}
         {\jRN{LMLA}}{47}{2000}{11--28}{#1}}
   \ITEE{#3}{TYTam2002}{
      \BIB{#2}{T.Y. Tam}
         {Generalized Schur-concave functions and Eaton triples}
         {\jRN{LMLA}}{50}{2002}{113--120}{#1}}
   \ITEE{#3}{TYTam,WCHill2001}{
      \BIB{#2}{T.Y. Tam and W.C. Hill}
         {On $G$-invariant norms}
         {\jRN{LAA}}{331}{2001}{101--112}{#1}}
   \ITEE{#3}{AFTiman,IAVestfrid1983}{
      \BIB{#2}{A.F. Timan and I.A. Vestfrid}
         {Any separable ultrametric space can be isometrically imbedded in $l_2$}
         {\jRN{FAA}}{17}{1983}{70--71}{#1}}
   \ITEE{#3}{JTomiyama1958}{
      \BIB{#2}{J. Tomiyama}
         {Generalized dimension function for $\WWw^*$-algebras of infinite type}
         {\jRN{TohokuMJ} (2)}{10}{1958}{121--129}{#1}}
   \ITEE{#3}{HTorunczyk1970}{
      \BIB{#2}{H. Toru\'{n}czyk}
         {Remarks on Anderson's paper ``On topological infinite deficiency''}
         {\jRN{FM}}{66}{1970}{393--401}{#1}}
   \ITEE{#3}{HTorunczyk1970a}{
      \BIb{#2}{H. Toru\'{n}czyk}
         {$G$-$K$-absorbing and skeletonized sets in metric spaces}
         {Ph.D. thesis, Inst. Math. Polish Acad. Sci., Warszawa}{1970}{#1}}
   \ITEE{#3}{HTorunczyk1972}{
      \BIB{#2}{H. Toru\'{n}czyk}
         {A short proof of Hausdorff's theorem on extending metrics}
         {\jRN{FM}}{77}{1972}{191--193}{#1}}
   \ITEE{#3}{HTorunczyk1974}{
      \BIB{#2}{H. Toru\'{n}czyk}
         {Absolute retracts as factors of normed linear spaces}
         {\jRN{FM}}{86}{1974}{53--67}{#1}}
   \ITEE{#3}{HTorunczyk1975}{
      \BIB{#2}{H. Toru\'{n}czyk}
         {On Cartesian factors and the topological classification of linear metric spaces}
         {\jRN{FM}}{88}{1975}{71--86}{#1}}
   \ITEE{#3}{HTorunczyk1978}{
      \BIB{#2}{H. Toru\'{n}czyk}
         {Concerning locally homotopy negligible sets and characterization of $l_2$-manifolds}
         {\jRN{FM}}{101}{1978}{93--110}{#1}}
   \ITEE{#3}{HTorunczyk1980}{
      \BiB{#2}{H. Toru\'{n}czyk}{Characterization of infinite-dimensional manifolds}{in:}
         {Proceedings of the International Conference on Geometric Topology (Warsaw, 1978)}
         {\jRN{PWN}}{1980}{431--437}{#1}}
   \ITEE{#3}{HTorunczyk1981}{
      \BIB{#2}{H. Toru\'{n}czyk}
         {Characterizing Hilbert space topology}
         {\jRN{FM}}{111}{1981}{247--262}{#1}}
   \ITEE{#3}{HTorunczyk1985}{
      \BIB{#2}{H. Toru\'{n}czyk}
         {A correction of two papers concerning Hilbert manifolds}
         {\jRN{FM}}{125}{1985}{89--93}{#1}}
   \ITEE{#3}{KTsuda1985}{
      \BIB{#2}{K. Tsuda}
         {A note on closed embeddings of finite dimensional metric spaces}
         {\jRN{BLondMS}}{17}{1985}{273--278}{#1}}
   \ITEE{#3}{PSUrysohn1925}{
      \BIB{#2}{P.S. Urysohn}
         {Sur un espace m\'{e}trique universel}
         {\jRN{CRASParis}}{180}{1925}{803--806}{#1}}
   \ITEE{#3}{PSUrysohn1927}{
      \BIB{#2}{P.S. Urysohn}
         {Sur un espace m\'{e}trique universel}
         {\jRN{BullSM}}{51}{1927}{43--64, 74--96}{#1}}
   \ITEE{#3}{VVUspenskij1986}{
      \BIB{#2}{V.V. Uspenskij}
         {A universal topological group with a countable basis}
         {\jRN{FAA}}{20}{1986}{86--87}{#1}}
   \ITEE{#3}{VVUspenskij1990}{
      \BIB{#2}{V.V. Uspenskij}
         {On the group of isometries of the Urysohn universal metric space}
         {\jRN{CMUC}}{31}{1990}{181--182}{#1}}
   \ITEE{#3}{VVUspenskij2004}{
      \BIB{#2}{V.V. Uspenskij}
         {The Urysohn universal metric space is homeomorphic to a Hilbert space}
         {\jRN{TopA}}{139}{2004}{145--149}{#1}}
   \ITEE{#3}{VVUspenskij2008}{
      \BIB{#2}{V.V. Uspenskij}
         {On subgroups of minimal topological groups}
         {\jRN{TopA}}{155}{2008}{1580--1606}{#1}}
   \ITEE{#3}{VSVaradarajan1963}{
      \BIB{#2}{V.S. Varadarajan}
         {Groups of automorphisms of Borel spaces}
         {\jRN{TAMS}}{109}{1963}{191--220}{#1}}
   \ITEE{#3}{AMVershik1998}{
      \BIB{#2}{A.M. Vershik}
         {The universal Urysohn space, Gromov's metric triples, and random metrics on the series of natural numbers}
         {\jRN{UspekhiMN}}{53}{1998}{57--64}{#1} English translation: \jRN{RussMS}{} \textbf{53} (1998), 921--928. 
         Correction: \jRN{UspekhiMN}{} \textbf{56} (2001), p. 207. English translation: \jRN{RussMS}{} \textbf{56} 
         (2001), p. 1015.}
   \ITEE{#3}{AMVershik2002}{
      \BIb{#2}{A.M. Vershik}
         {Random metric spaces and the universal Urysohn space}
         {Fundamental Mathematics Today. 10th anniversary of the Independent Moscow University. MCCME Publ.}{2002}{#1}}
   \ITEE{#3}{NWeaver1999}{
      \BIb{#2}{N. Weaver}
         {Lipschitz Algebras}
         {World Scientific}{1999}{#1}}
   \ITEE{#3}{JWeidmann1980}{
      \BIb{#2}{J. Weidmann}
         {Linear Operators in Hilbert Spaces}
         {(Graduate Texts in Mathematics, vol. 68) Springer-Verlag New York Inc.}{1980}{#1}}
   \ITEE{#3}{JEWest1969}{
      \BIB{#2}{J.E. West}
         {Approximating homotopies by isotopies in Fr\'{e}chet manifolds}
         {\jRN{BAMS}}{75}{1969}{1254--1257}{#1}}
   \ITEE{#3}{JEWest1969a}{
      \BIB{#2}{J.E. West}
         {Fixed-point sets of transformation groups on infinite-product spaces}
         {\jRN{PAMS}}{21}{1969}{575--582}{#1}}
   \ITEE{#3}{JEWest1970}{
      \BIB{#2}{J.E. West}
         {The ambient homeomorphy of infinite-dimensional Hilbert spaces}
         {\jRN{PacJM}}{34}{1970}{257--267}{#1}}
   \ITEE{#3}{JHCWhitehead1949}{
      \BIB{#2}{J.H.C. Whitehead}
         {Combinatorial homotopy I}
         {\jRN{BAMS}}{55}{1949}{213--245}{#1}}
   \ITEE{#3}{GTWhyburn1942}{
      \BIb{#2}{G. T. Whyburn}
         {Analytic Topology}
         {Amer. Math. Soc. Colloquium Publications (vol. XXVIII), New York}{1942}{#1}}
   \ITEE{#3}{WWogen1969}{
      \BIB{#2}{W. Wogen}
         {On generators for von Neumann algebras}
         {\jRN{BAMS}}{75}{1969}{95--99}{#1}}
   \ITEE{#3}{RYTWong1967}{
      \BIB{#2}{R.Y.T. Wong}
         {On homeomorphisms of certain infinite dimensional spaces}
         {\jRN{TAMS}}{128}{1967}{148--154}{#1}}
   \ITEE{#3}{LYang,JZhang1987}{
      \BIB{#2}{L. Yang and J. Zhang}
         {Average distance constants of some compact convex space}
         {\jRN{JChinUST}}{17}{1987}{17--23}{#1}}
   \ITEE{#3}{PZakrzewski1993}{
      \BIB{#2}{P. Zakrzewski}
         {The existence of invariant $\sigma$-finite measures for a group of transformations}
         {\jRN{IsraelJM}}{83}{1993}{275--287}{#1}}
   \ITEE{#3}{PZakrzewski2002}{
      \BIb{#2}{P. Zakrzewski}
         {Measures on Algebraic-Topological Structures, Handbook of Measure Thoery}
         {E. Pap, ed., Elsevier, Amsterdam}{2002, 1091--1130}{#1}}
   \ITEE{#3}{KZhu2000}{
      \BIB{#2}{K. Zhu}
         {Operators in Cowen-Douglas classes}
         {\jRN{IllinoisJM}}{44}{2000}{767--783}{#1}}
   }

\newcommand{\mypaplist}[2][]{
   \ITEE{#2}{pn1}{
      \myBIB{Separate and joint similarity to families of normal operators}
         {\jRN[#1]{SM}}{149}{2002}{39--62}}
   \ITEE{#2}{pn2}{
      \myBIB{Locally arcwise connected metrizable spaces with the fixed point property are complete-metrizable}
         {\jRN[#1]{TopA}}{153}{2006}{1639--1642}}
   \ITEE{#2}{pn3}{
      \myBIB{Invariant measures for equicontinuous semigroups of continuous transformations of a compact Hausdorff space}
         {\jRN[#1]{TopA}}{153}{2006}{3373--3382}}
   \ITEE{#2}{pn4}{
      \myBIB{Approximation of the Hausdorff distance by the distance of continuous surjections}
         {\jRN[#1]{TopA}}{154}{2007}{655--664}}
   \ITEE{#2}{pn5}{
      \myBIB{Generalized Haar integral}
         {\jRN[#1]{TopA}}{155}{2008}{1323--1328}}
   \ITEE{#2}{pn6}{
      \myBIB{Integration and Lipschitz functions}
         {\jRN[#1]{RCMP}}{57}{2008}{391--399}}
   \ITEE{#2}{pn7}{
      \myBIB{Canonical Banach function spaces generated by Urysohn universal spaces. Measures as Lipschitz maps}
         {\jRN[#1]{SM}}{192}{2009}{97--110}}
   \ITEE{#2}{pn8}{
      \myBIB{Urysohn universal spaces as metric groups of exponent $2$}
         {\jRN[#1]{FM}}{204}{2009}{1--6}}
   \ITEE{#2}{pn9}{
      \myBIB{Central subsets of Urysohn universal spaces}
         {\jRN[#1]{CMUC}}{50}{2009}{445--461}}
   \ITEE{#2}{pn10}{
      \myBIB[P. Niemiec and T.Y. Tam]{A representation of $G$-in\-variant norms for Eaton triple}
         {\jRN[#1]{JCA}}{18}{2011}{59--65}}
   \ITEE{#2}{pn11}{
      \myBIB{Functor of extension of contractions on Urysohn universal spaces}
         {\jRN[#1]{ACS}}{}{2009}{\texttt{DOI: 10.1007/s10485-009-9218-z}}}
   \ITEE{#2}{pn12}{
      \myBIB{Ultra-$\mM$-separability}
         {\jRN[#1]{TopA}}{157}{2010}{669--673}}
   \ITEE{#2}{pn13}{
      \myBIB{Functor of extension of $\Lambda$-isometric maps between central subsets 
         of the unbounded Urysohn universal space}{\jRN[#1]{CMUC}}{51}{2010}{541--549}}
   \ITEE{#2}{pn14}{
      \myBIB{Normed topological pseudovector groups}{\jRN[#1]{ACS}}{}{2010}
         {\ITE{\equal{#1}{}}{\texttt{DOI: 10.1007/s10485\-010-9239-7}}{\texttt{DOI: 10.1007/s10485-010-9239-7}}}}
   \ITEE{#2}{pn15}{
      \myBIB{Topological structure of Urysohn universal spaces}
         {\jRN[#1]{TopA}}{158}{2011}{352--359}}
   \ITEE{#2}{pn16}{
      \myBIB{A note on invariant measures}
         {\jRN[#1]{OpusM}}{31}{2011}{425--431}}
   \ITEE{#2}{pn17}{
      \myBIB{Strengthened Stone-Weierstrass type theorem}
         {\jRN[#1]{OpusM}}{31}{2011}{645--650}}
   \ITEE{#2}{pnX2}{
      \myBAPP{Functor of continuation in Hilbert cube and Hilbert space}
         {to appear in \jRN[#1]{FM}}}
   \ITEE{#2}{pnX3}{
      \myBAPP{Norm closures of orbits of bounded operators}
         {to appear.}}
   \ITEE{#2}{pnX6}{
      \myBAPP{Extending maps by injective $\sigma$-$Z$-maps in Hilbert manifolds}
         {to appear in \jRN[#1]{BullPol}}}
   \ITEE{#2}{pnX7}{
      \myBAPP{Spaces of measurable functions}
         {submitted to \jRN[#1]{CollectM}}}
   \ITEE{#2}{pnX8}{
      \myBAPP{Normal systems over ANR's, rigid embeddings and nonseparable absorbing sets}
         {submitted to \jRN[#1]{ActaMSinES}}}
   \ITEE{#2}{pnX9}{
      \myBAPP{Borel structure of the spectrum of a closed operator}
         {submitted to \jRN[#1]{SM}}}
   \ITEE{#2}{pnX10}{
      \myBAPP{Central points and measures and dense subsets of compact metric spaces}
         {submitted to \jRN[#1]{TopMethNA}}}
   \ITEE{#2}{pnX11}{
      \myBAPP{Generalized absolute values and polar decompositions of a bounded operator}
         {submitted to \jRN[#1]{IEOT}.}}
   \ITEE{#2}{pnX12}{
      \myBAPP{Ultrametrics, extending of Lipschitz maps and nonexpansive selections}
         {accepted for publication in \jRN[#1]{HJM}}}
   \ITEE{#2}{pnX13}{
      \myBAPP{A note on ANR's}
         {submitted to \jRN[#1]{TopA}}}
   \ITEE{#2}{pnX14}{
      \myBAPP{Problem with almost everywhere equality}
         {submitted to \jRN[#1]{ArchM}}}
   \ITEE{#2}{pnX15}{
      \myBAPP{Universal valued Abelian groups}
         {submitted to \jRN[#1]{LNM}}}
   \ITEE{#2}{pnX16}{
      \myBAPP{Unitary equivalence and decompositions of finite systems of closed densely defined operators 
         in Hilbert spaces}{submitted to \jRN[#1]{DissM}}}
   }
